\documentclass[11pt]{article}
\usepackage{amsfonts,amsmath,amssymb}
\usepackage{amsthm}
  \usepackage{txfonts}
 \usepackage[colorlinks=true]{hyperref}
 \hypersetup{urlcolor=blue, citecolor=red}
\newtheorem{remark}{Remark}[section]
\newtheorem{theorem}{Theorem}[section]
\newtheorem{lemma}[theorem]{Lemma}
\newtheorem{definition}[theorem]{Definition}
\numberwithin{equation}{section}

\begin{document}
\voffset=-1 true cm \setcounter{page} {1}
\title{Properties of the minimizers for a constrained minimization problem arising in Kirchhoff equation  }
\author{Helin Guo, Huan-Song Zhou\thanks{Corresponding
author.\newline Email address: H.L.Guo: qfguohelin@126.com; H.S.Zhou: hszhou@whut.edu.cn. \hfill
\newline This work was supported by  NFSC under Grants No. 11871387,  11931012 and partially by the fundamental Research Funds for the Central Universities(WUT: 2019IVA106, 2019IB009).}\\
\small Center for Mathematical Sciences and Department of Mathematics, \\
\small School of Science, Wuhan University of Technology, Wuhan, 430070, P. R. China
}
\date{}
\maketitle
{\bf Abstract}\quad Let $a>0,b>0$ and $V(x)\geq0$ be a coercive function in $\mathbb R^2$. We study the following constrained minimization problem on a suitable weighted Sobolev space $\mathcal{H}$:
\begin{equation*}
e_{a}(b):=\inf\left\{E_{a}^{b}(u):u\in\mathcal{H}\ \mbox{and}\ \int_{\mathbb R^{2}}|u|^{2}dx=1\right\},
\end{equation*}
where $E_{a}^{b}(u)$ is a Kirchhoff type energy functional defined on $\mathcal{H}$ by
\begin{equation*}
E_{a}^{b}(u)=\frac{1}{2}\int_{\mathbb R^{2}}[|\nabla u|^{2}+V(x)u^{2}]dx+\frac{b}{4}\left(\int_{\mathbb R^{2}}|\nabla u|^{2}dx\right)^{2}-\frac{a}{4}\int_{\mathbb R^{2}}|u|^{4}dx.
\end{equation*}
It is known that, for some $a^{\ast}>0$, $e_{a}(b)$ has no minimizer if $b=0$ and $a\geq a^{\ast}$, but $e_{a}(b)$ has always a minimizer for any $a\geq0$ if $b>0$. The aim of this paper is to investigate the limit behaviors of the minimizers of $e_{a}(b)$ as $b\rightarrow0^{+}$. Moreover, the uniqueness of the minimizers of $e_{a}(b)$ is also discussed for $b$ close to 0.
\section{Introduction}
\quad In this paper, we are concerned with the following constrained minimization problem:
\begin{equation}\label{eq1.2}
e_{a}(b):=\inf\left\{E_{a}^{b}(u):u\in\mathcal{H}\ \mbox{and}\ \|u\|_{2}^{2}\triangleq\int_{\mathbb R^{2}}|u|^{2}dx=1\right\},
\end{equation}
where $\mathcal{H}$ is a weighted Sobolev space given by
\begin{equation*}
\mathcal{H}\triangleq\Big\{u\in H^{1}(\mathbb R^{2}):\int_{\mathbb R^{2}}V(x)u^{2}dx<\infty\Big\}\ \mbox{for some nonnegative}\ V(x)\in L_{loc}^{\infty}(\mathbb R^2),
\end{equation*}
and $E_{a}^{b}(u)$ is a Kirchhoff type energy functional as follows
\begin{equation}\label{eq1.2*}
E_{a}^{b}(u)=\frac{1}{2}\int_{\mathbb R^{2}}[|\nabla u|^{2}+V(x)u^{2}]dx+\frac{b}{4}\left(\int_{\mathbb R^{2}}|\nabla u|^{2}dx\right)^{2}-\frac{a}{4}\int_{\mathbb R^{2}}|u|^{4}dx,\ u\in \mathcal{H},
\end{equation}
$a$ and $b$ are positive parameters.

The above minimization problem arises in studying the following elliptic eigenvalue problem
\begin{equation}\label{eq1.1}
-\left(1+b\int_{\mathbb R^{2}}|\nabla u|^{2}dx\right)\Delta u+V(x)u=a|u|^{2}u+\mu u,\ \ x\in \mathbb{R}^{2},\ \mu\in\mathbb R,
\end{equation}
which is essentially a stationary (time independent) Kirchhoff equation, see e.g., \cite{ap,as1,K} for more backgrounds.
For bounded $V(x)$,
problem \eqref{eq1.1} had been studied in many papers, see e.g., \cite{HL,HZ,LY} and the references therein.

It is known that a minimizer of problem \eqref{eq1.2} corresponds a solution of \eqref{eq1.1} with $\mu$ being a suitable Lagrange multiplier. When $b=0$, \eqref{eq1.1} with  given $\|u\|_{2}$ becomes the famous Gross-Pitaevskii (GP) equation (time independent case) which is important in the study of Bose-Einstein condensates (BEC), see, e.g., \cite{DG}. For this reason, problem \eqref{eq1.2} or \eqref{eq1.1} with $b=0$ has received a lot of interest in mathematics in recent years, see e.g.,
\cite{Bao2,GS,GLW,GZZ,GZZ3,GZ,Z1,Z}, provided $V(x)$ is a coercive potential, that is,
\begin{equation}\label{eq1.4}
V(x)\in L_{loc}^{\infty}(\mathbb R^{2},\mathbb R^{+}),\ \lim_{|x|\rightarrow\infty}V(x)=\infty\ \mbox{and}\ \inf_{x\in\mathbb R^{2}}V(x)=0.
\end{equation}
So, in order to compare more clearly the results of Kirchhoff problem \eqref{eq1.2} ($b>0$) with that of GP equations  \cite{Bao2,GS}, that is, $b=0$ in \eqref{eq1.2}, in this paper we only consider problem \eqref{eq1.2} in $\mathbb{R}^2$, but it is not difficulty to extend our results of the paper to $\mathbb{R}^n \ (n\geq 2)$ by using the results of \cite{GZZ2}.

Under \eqref{eq1.4}, in \cite{Bao2,GS,Z}, the authors proved that problem \eqref{eq1.2} with $b=0$ has a minimizer if
$a\in[0,a^{\ast})$ and has no minimizer if $a\geq a^{\ast}$, where $a^{\ast}=\|Q\|_{2}^{2}$ and $Q(x)$ is the unique (up to translations) radially symmetric positive solution of the equation 
\begin{equation}\label{eq1.5}
-\Delta u+u=u^{3},\ u\in H^{1}(\mathbb R^{2}).
\end{equation}
Moreover, the concentration and symmetry breaking of minimizers were also studied in \cite{GS,GZZ3} as $a\nearrow a^{\ast}$ under different types of trapping potential, and the uniqueness of minimizers was proved in \cite{GLW} as $a$ close to $a^{\ast}$. But, when $b\neq0$, it was proved in a very recent paper \cite{GZZ2} that \eqref{eq1.2} has always a minimizer for all $a\geq0$ and $b>0$, that is, for each $b>0$ there is a minimizer for $e_{a}(b)$. Therefore, a nature question is what would happened if $b\rightarrow0^{+}?$ Intuitively, we may expect that the minimizers of $e_{a}(b)(b>0)$ should converge to a minimizer of $e_{a}(0)$ (i.e., $e_{a}(b)\ \mbox{with}\ b=0$). However, this may not be true at least for $a\geq a^{\ast}$, because $e_{a}(0)$ has no minimizer if $a\geq a^{\ast}$.

The aim of this paper is to give some detailed information on the limit behavior of the minimizers of $e_{a}(b)(b>0)$ as $b\to0^{+}$. Moreover, we are also interested in the uniqueness of minimizers of $e_{a}(b)$ with $b>0$ being small enough and any given $a\in[a^{\ast},+\infty)$. However, due to the presence of the nonlocal term $(\int_{\mathbb R^{2}}|\nabla u|^{2}dx)^{2}$ in \eqref{eq1.2*}, the methods used in \cite{GS,GZZ3} can not be followed directly in our case. Particularly, in discussing the uniqueness of the minimizers for $e_{a}(b)$ with $b>0$ and close to $0$, we have to encounter much more complicated and technical calculations than in \cite{GLW}. To overcome these difficulties, we need to use some new ideas in getting the energy estimates and proving the uniqueness of $e_{a}(b)$. Before giving the main results of the paper, we introduce the following auxiliary functional
\begin{equation}\label{eq1.6*}
\overline{E}_{a}^{b}(u)=\frac{1}{2}\int_{\mathbb R^{2}}|\nabla u|^{2}dx+\frac{b}{4}\left(\int_{\mathbb R^{2}}|\nabla u|^{2}dx\right)^{2}-\frac{a}{4}\int_{\mathbb R^{2}}|u|^{4}dx,
\end{equation}
and the constrained minimization problem
\begin{equation}\label{eq1.6}
\overline{e}_{a}(b):=\inf\left\{\overline{E}_{a}^{b}(u):u\in H^{1}(\mathbb R^2)\ \mbox{and}\ \int_{\mathbb R^{2}}|u|^{2}dx=1\right\}.
\end{equation}
When $b>0$, it was proved in \cite[Theorem 1.1]{GZZ2} that $\overline{e}_{a}(b)$ in \eqref{eq1.6} has a minimizer if and only if $a>a^{\ast}$, but $e_{a}(b)$ in \eqref{eq1.2} has always a minimizer for any $a\geq0$, see, e.g.,\cite[Theorem 1.2]{GZZ2}.
Since $|\nabla u|=|\nabla|u||$ holds for a.e.$x\in\mathbb R^2$, without loss of generality, we always suppose that minimizers of $e_{a}(b)$ and $\overline{e}_{a}(b)$ are nonnegative. Now, we state our results as follows.
\begin{theorem}\label{th1}
Suppose that $V(x)$ satisfies \eqref{eq1.4} and $V(x)\in C^\alpha_{loc}(\mathbb{R}^2)$ with some $\alpha\in(0,1)$. For any given $a\geq a^\ast$, let $u_{k}$ be a nonnegative minimizer of $e_{a}(b_{k})$ with $b_{k}\xrightarrow{k\to\infty}0^{+}$. Then, there exists a subsequence of $\{u_{k}\}$, still denoted by $\{u_{k}\}$, such that each $u_{k}$ has a unique global maximum point $z_{k}$ satisfying
\begin{equation}\label{eq1.7}
\lim_{k\rightarrow\infty}z_{k}=x_{0}\ \mbox{with}\ x_{0}\in\mathbb R^{2}\ \mbox{and}\ V(x_{0})=0,
\end{equation}
\begin{equation}\label{eq1.8}
\lim_{k\rightarrow+\infty}\epsilon_{k}u_{k}(\epsilon_{k}x+z_{k})=\frac{Q(x)}{\sqrt{a^{\ast}}} \ \mbox{in} \ H^{1}(\mathbb R^{2}),
\end{equation}
where $Q(x)$ is the unique positive solution of \eqref{eq1.5}, and  $\epsilon_{k}\xrightarrow{k\to\infty}0^+$ which is given by
\begin{equation}\label{eq1.9}
\epsilon_{k}=
\begin{cases}
\left(\int_{\mathbb R^2}|\nabla u_{k}|^2dx\right)^{-\frac{1}{2}}\quad \mbox{for}\ a=a^\ast,\\
\left(\frac{b_{k}a^{\ast}}{a-a^{\ast}}\right)^{\frac{1}{2}}\quad\quad\quad\quad\  \mbox{for}\ a>a^\ast.
\end{cases}
\end{equation}
Moreover, if $a>a^{\ast}$,
 \begin{equation}\label{eq1.10}
   e_{a}(b_{k})=-\frac{1}{4b_{k}}\left(\frac{a-a^{\ast}}{a^{\ast}}\right)^{2}(1+o(1))\ \mbox{as}\ k\rightarrow\infty.
   \end{equation}
\end{theorem}
\begin{remark}\label{remark}
If $V(x)$ satisfies \eqref{eq1.4}, it is known by \cite[Theorem 2.1]{BW} or \cite[Theorem XIII.67]{RS} that the embedding from $\mathcal{H}$ into $L^{s}(\mathbb R^{2})(2\leq s<\infty)$ is compact. Hence, for $a\in[0,a^{\ast})$, similar to the proof of Theorem 1.4 in \cite{GZZ2}, we know that a minimizer $u_{b}$ of \eqref{eq1.2} must converge to a minimizer $u_{0}$ of $e_{a}(0)$ as $b\rightarrow 0^{+}$.
\end{remark}

For $a\geq a^\ast$, Theorem \ref{th1} shows that the nonnegative minimizers of \eqref{eq1.2} concentrate at a global minimum point of $V(x)$ as $b\rightarrow0^{+}$. But, under the general coercive condition \eqref{eq1.4}, it seems impossible to have more detailed information about the location of $x_{0}$ and the blowup rates of $\epsilon_{k}$. Motivated by \cite{CLL,GS,GZZ,GZZ3}, in what follows, we give some additional assumptions on $V(x)$, with which we may refine the results of Theorem \ref{th1} by establishing the optimal energy estimates of $e_{a}(b)$.
\begin{definition}\label{de1}
A function $f(x)$ is called homogeneous of degree $q\in\mathbb R^{+}$(about the origin) if there exists some $q>0$ such that
\begin{equation*}
f(tx)=t^{q}f(x),\ \mbox{in}\ \mathbb R^{2}\ \mbox{for\ any}\ t>0.
\end{equation*}
The above definition implies that if $f(x)\in C(\mathbb R^{2},\mathbb R^+)$ is homogeneous of degree $q>0$, then
\begin{equation*}
0\leq f(x)\leq C|x|^q\ \mbox{in}\ \mathbb R^2,
\end{equation*}
where $C$ denotes the maximum of $f(x)$ on $\partial B_{1}(0)$. Moreover, if $f(x)\rightarrow\infty\ as\ |x|\rightarrow\infty$, then 0 is the unique minimum point of $f(x)$.
\end{definition}
Inspired by \cite{GS}, we assume that $V(x)$ has exactly $m$ global minimum points, namely
\begin{equation}\label{eq1.13}
Z:=\{x\in\mathbb R^2:V(x)=0\}=\{x_1,x_2,...x_m\},\ \mbox{where}\ m\geq1.
\end{equation}
We then assume that $V(x)$ is almost homogeneous of degree $p_i>0$ around $x_i$, i.e., there exists some $V_i(x)\in C_{loc}^2(\mathbb R^2)$ satisfying $\displaystyle\lim_{|x|\rightarrow\infty}V_i(x)=\infty$, which is homogeneous of degree $p_i>0$, such that
\begin{equation}\label{eq1.14}
\lim_{x\rightarrow0}\frac{V(x+x_i)}{V_i(x)}=1,\ i=1,2,...m.
\end{equation}
Additionally, we define $H_i(y)$ by
\begin{equation}\label{eq1.15}
H_i(y):=\int_{\mathbb R^2}V_i(x+y)Q^2(x)dx,\ i=1,2,...m.
\end{equation}
Set
\begin{equation}\label{eq1.16}
p:=\max_{1\leq i\leq m}p_i,\ \mbox{and}\ \overline{Z}:=\{x_i\in Z:p_i=p\}\subset Z.
\end{equation}
Define
\begin{equation}\label{eq1.17}
\lambda_0:=\min_{i\in \Lambda}\lambda_i,\ \mbox{where}\ \lambda_i:=\min_{y\in\mathbb R^2}H_i(y)\ \mbox{and}\ \Lambda:=\{i:x_i\in\overline{Z}\}
\end{equation}
Denote
\begin{equation}\label{eq1.18}
Z_0:=\{x_i\in\overline{Z}:\lambda_i=\lambda_0\}
\end{equation}
the set of the flattest minimum points of $V(x)$. Under the above assumptions, our following theorem gives a precise description on the concentration behavior of the minimizers of \eqref{eq1.2} as $b\rightarrow0^+$.
\begin{theorem}\label{th2}
Assume that $V(x)$ satisfies conditions \eqref{eq1.4} and \eqref{eq1.14}. For $a\geq a^{\ast}$, let $u_k$ be a nonnegative minimizer of $e_{a}(b_k)$ as in Theorem \ref{th1} with $b_{k}\xrightarrow{k\to\infty}0^{+}$ and $z_{k}$ be the unique maximum point of $u_{k}$. Then,
\begin{equation}\label{eq1.19}
\lim_{k\rightarrow+\infty}\overline{\epsilon}_{k}u_{k}(\overline{\epsilon}_{k}x+z_{k})=\frac{Q(x)}{\sqrt{a^{\ast}}} \ \mbox{in} \ H^{1}(\mathbb R^{2}),
\end{equation}
where $\overline{\epsilon}_{k}$ is given by
\begin{equation}\label{eq1.20}
\overline{\epsilon}_{k}:=
\begin{cases}
\left(\frac{2b_{k}a^\ast}{p\lambda_0}\right)^{\frac{1}{p+4}} \ \ \quad\quad\mbox{for}\ a=a^\ast,\\
\epsilon_k=\left(\frac{b_{k}a^{\ast}}{a-a^{\ast}}\right)^{\frac{1}{2}}\quad\ \mbox{for}\ a>a^\ast,
\end{cases}
\end{equation}
and $z_{k}$ satisfies
\begin{equation}\label{eq1.21**}
\lim_{k\rightarrow\infty}\frac{z_{k}-x_{0}}{\overline{\epsilon}_{k}}=y_0
\end{equation}
with $x_{0}=x_{i_0}\in Z_0$ for some $1\leq i_{0}\leq m$, and $y_{0}\in\mathbb R^2$ satisfying $H_{i_0}(y_0)=\displaystyle\min_{y\in\mathbb R^2}H_{i_0}(y)=\lambda_0$.

Moreover, if $a=a^\ast$,
\begin{equation}\label{eq1.21*}
\lim_{k\rightarrow\infty}\frac{e_{a^\ast}(b_{k})}{b_{k}^{\frac{p}{p+4}}}=\frac{4+p}{4p}\left(\frac{p\lambda_{0}}{2a^\ast}\right)^{\frac{4}{p+4}}.
\end{equation}
\end{theorem}
Theorem \ref{th2} shows that the nonnegative minimizers of $e_{a}(b)$ must concentrate at one of the flattest global minimum point of $V(x)$, as $b\rightarrow0^+$. Different from the discussions in \cite{GS}, in our case the nonlocal term $(\int_{\mathbb R^{2}}|\nabla u|^{2}dx)^{2}$ causes some new difficulties in analyzing the asymptotic behavior of the nonnegative minimizers for $e_{a}(b)$.

Finally, we are concerned with the uniqueness of the minimizers of $e_{a}(b)$ as $b$ close to 0, under some further assumptions of $V(x)$. Motivated by the uniqueness results addressed in \cite{GLW}, We assume that $V(x)$ has a unique flattest global minimum point, i.e., $Z_0$ defined in \eqref{eq1.18} contains only one element. Our uniqueness results can be stated as follows.
\begin{theorem}\label{th4}
Suppose that $V(x)\in C^{2}(\mathbb R^2)$ satisfies \eqref{eq1.4} and \eqref{eq1.14}. Let $Z_0$ in \eqref{eq1.18} have only one point $x_{1}$, and
\begin{equation}\label{eq1.21}
y_0\ \mbox{is the unique and non-degenerate critical point of}\ H_1(y)\ \mbox{defined by}\ \eqref{eq1.15}.
\end{equation}
If there exist $\beta>0$ and $R_0>0$ such that
\begin{equation}\label{eq1.22}
V(x)\leq Ce^{\beta|x|},\ \mbox{if}\ |x|\ \mbox{is large},
\end{equation}
and
\begin{equation}\label{eq1.23}
\frac{\partial V(x+x_1)}{\partial x_j}=\frac{\partial V_1(x)}{\partial x_j}+W_j(x)\ \mbox{and}\ |W_j(x)|\leq C|x|^{s_j}\ \mbox{in}\ B_{R_0}(0),
\end{equation}
where $s_j>p-1$ for $j=1,2$. Then, for $a\geq a^{\ast}$, there exists a unique nonnegative minimizer for $e_{a}(b)$ as $b>0$ being small enough.
\end{theorem}
This paper is organized as follows. In Section 2, $\overline{e}_{a}(b)$ is calculated, then the relation between $e_{a}(b)$ and $\bar{e}_{a}(b)$ is established as $b\rightarrow0^+$, and the proof of Theorem \ref{th1} is finally given under general coercive potential $V(x)$ in \eqref{eq1.4}. In Section 3, Theorem \ref{th2} is proved based on some detailed energy estimates of $e_{a}(b)$. In Section 4, we prove the uniqueness of the minimizers for $e_{a}(b)$ by contradiction and some techniques on the local Pohozaev identities.
\section{Concentration behavior under general coercive potential.}

First of all, we recall the following Gagliardo-Nirenberg inequality \cite{W2}
\begin{equation}\label{eq2.1}
\int_{\mathbb R^{2}}|u|^{4}dx\leq \frac{2}{a^\ast}\int_{\mathbb R^{2}}|\nabla u|^{2}dx
\int_{\mathbb R^{2}}|u|^{2}dx,\ u\in H^{1}(\mathbb R^{2}),
\end{equation}
where the equality holds when $u=Q(x)$, the unique positive solution of \eqref{eq1.5}. Moreover, it follows from \eqref{eq1.5} and \eqref{eq2.1} that
\begin{equation}\label{eq2.2}
a^{\ast}=\int_{\mathbb R^{2}}|Q|^{2}dx=\int_{\mathbb R^{2}}|\nabla Q|^{2}dx=\frac{1}{2}\int_{\mathbb R^{2}}|Q|^{4}dx,
\end{equation}
and from Proposition 4.1 of \cite{GNN} that
\begin{equation}\label{eq2.3}
Q(x),\ |\nabla Q(x)|=O(|x|^{-\frac{1}{2}}e^{-|x|})\ \mbox{as}\ |x|\rightarrow\infty.
\end{equation}
\begin{lemma}\label{lemma2.1}
For any given $a>a^\ast$, by the definition of \eqref{eq1.6}, we have
\begin{equation}\label{eq2.4}
\overline{e}_{a}(b)=-\frac{1}{4b}\left(\frac{a-a^{\ast}}{a^{\ast}}\right)^{2},
\end{equation}
and the unique(up to translations) nonnegative minimizer of $\overline{e}_{a}(b)$ must be of the form
\begin{equation}\label{eq2.5}
\overline{u}_{b}(x)=\frac{r_{b}^{\frac{1}{2}}}{\sqrt{a^{\ast}}}Q(r_{b}^{\frac{1}{2}}x),\ \mbox{where}\ r_{b}=\frac{a-a^{\ast}}{ba^{\ast}}.
\end{equation}
\end{lemma}
\begin{proof}[\textup{\textbf{Proof}}]
For any $u\in H^1(\mathbb R^2)$ satisfying $\int_{\mathbb R^2}|u|^2dx=1$, it follows from \eqref{eq1.6*} and \eqref{eq2.1} that
\begin{equation*}
\overline{E}_{a}^{b}(u)\geq\frac{b}{4}\left(\int_{\mathbb R^{2}}|\nabla u|^{2}dx\right)^{2}-\frac{a-a^\ast}{2a^\ast}\int_{\mathbb R^{2}}|\nabla u|^{2}dx.
\end{equation*}
Set
\begin{equation}\label{eq2.6}
h(r)=\frac{b}{4}r^2-\frac{a-a^{\ast}}{2a^\ast}r,\ r\in[0,+\infty).
\end{equation}
By simple calculation, we know that $h(r)$ attains its global minimum at $r_{b}=\frac{a-a^{\ast}}{ba^{\ast}}$, hence
\begin{equation*}
\overline{E}_{a}^{b}(u)\geq h(r_{b})=-\frac{1}{4b}\left(\frac{a-a^{\ast}}{a^{\ast}}\right)^{2}.
\end{equation*}
This implies that
\begin{equation}\label{eq2.7}
\overline{e}_{a}(b)\geq h(r_{b})=-\frac{1}{4b}\left(\frac{a-a^{\ast}}{a^{\ast}}\right)^{2}.
\end{equation}

On the other hand, take $u_{t}(x)=\frac{t}{\sqrt{a^\ast}}Q(tx)(t>0)$, then $\int_{\mathbb R^2}|u_{t}|^2dx=1$, it follows from \eqref{eq2.2} that
\begin{equation} \label{eq2.8}
\int_{\mathbb R^{2}}|\nabla u_{t}|^{2}dx
=\frac{t^{2}\int_{\mathbb R^{2}}|\nabla Q|^{2}dx}{a^\ast}=t^{2},
\end{equation}
and
\begin{equation} \label{eq2.9}
\int_{\mathbb R^{2}}|u_{t}|^{4}dx
=\frac{t^{2}\int_{\mathbb R^{2}}|Q|^{4}dx}{(a^\ast)^{2}}
=\frac{2t^2}{a^{\ast}}.
\end{equation}
Hence,
\begin{equation*}
\overline{e}_{a}(b)\leq \overline{E}_{a}^{b}(u_{t})=\frac{b}{4}t^4-\frac{a-a^{\ast}}{2a^\ast}t^2=h(t^{2}),
\end{equation*}
where $h(\cdot)$ is given by \eqref{eq2.6}. Therefore, let $t=r_{b}^{\frac{1}{2}}$, we see that
\begin{equation}\label{eq2.9*}
\overline{e}_{a}(b)\leq h(r_{b})=-\frac{1}{4b}\left(\frac{a-a^{\ast}}{a^{\ast}}\right)^{2},
\end{equation}
this and \eqref{eq2.7} imply that \eqref{eq2.4} holds. Moreover, $\overline{e}_{a}(b)$ is attained by $\overline{u}_{b}(x)=\frac{r_{b}^{\frac{1}{2}}}{\sqrt{a^\ast}}Q(r_{b}^{\frac{1}{2}}x)$, and the proof is completed by the uniqueness (e.g., Theorem 1.1 in \cite{ZZ1}) of positive minimizer for $\overline{e}_{a}(b)$.
\end{proof}
\begin{lemma}\label{lemma2.2}
For any given $a>a^{\ast}$, let $V(x)$ satisfy \eqref{eq1.4} and let $u_{b}$ be a nonnegative minimizer of $e_{a}(b)$. Then,
\begin{equation}\label{eq2.10}
0\leq e_{a}(b)-\overline{e}_{a}(b)\rightarrow0\ \mbox{as}\ b\rightarrow0^+,
\end{equation}
and
\begin{equation}\label{eq2.11}
\int_{\mathbb R^2}V(x)|u_{b}|^{2}dx\rightarrow0\ \mbox{as}\ b\rightarrow0^+.
\end{equation}
\end{lemma}
\begin{proof}[\textup{\textbf{Proof}}]
By the definition of $e_{a}(b)$ and $\overline{e}_{a}(b)$, it is easy to see that
\begin{equation*}
e_{a}(b)-\overline{e}_{a}(b)\geq0.
\end{equation*}
Now, we turn to giving an upper bound for $e_{a}(b)-\overline{e}_{a}(b)$. Let $0\leq\xi(x)\in C_{0}^{\infty}(\mathbb R^{2})$ be a cut-off function such that
\begin{equation}\label{eq2.12}
\xi(x)\equiv1\ \mbox{if}\ |x|\leq1,\ \xi(x)\equiv0\ \mbox{if}\ |x|\geq2,\ \mbox{and}\ 0\leq\xi(x)\leq1,\ \mbox{if}\ 1\leq|x|\leq2.
\end{equation}
For any $x_{0}\in\mathbb R^2$, set
\begin{equation}\label{eq2.13}
\widehat{u}_{b}(x)=A_{b}\xi(x-x_{0})\overline{u}_{b}(x-x_{0}),
\end{equation}
where $\overline{u}_{b}(x)$ is defined in \eqref{eq2.5}, and $A_{b}>0$ is chosen so that $\int_{\mathbb R^2}|\widehat{u}_{b}|^2dx=1$. By the exponential decay property of $Q(x)$ in \eqref{eq2.3} and the definition of \eqref{eq2.5}, we have
\begin{equation}\label{eq2.14}
0\leq A_{b}^{2}-1=\frac{\int_{\mathbb R^2}[1-\xi^{2}(r_{b}^{-\frac{1}{2}}x)]Q^{2}(x)dx}{\int_{\mathbb R^2}\xi^{2}(r_{b}^{-\frac{1}{2}}x)Q^{2}(x)dx}
\leq Ce^{-2r_{b}^{\frac{1}{2}}}\ \mbox{as}\ b\rightarrow0^+.
\end{equation}
Then,
\begin{equation}\label{eq2.16}
\aligned
\int_{\mathbb R^2}|\widehat{u}_{b}|^{4}dx&=\frac{A_{b}^{4}r_{b}}{(a^\ast)^{2}}\int_{\mathbb R^2}\xi^{4}(r_{b}^{-\frac{1}{2}}x)Q^{4}(x)dx
\geq\frac{r_{b}}{(a^\ast)^{2}}\int_{\mathbb R^2}Q^{4}(x)dx-Ce^{-2r_{b}^{\frac{1}{2}}}\\
&=\int_{\mathbb R^2}|\overline{u}_{b}|^{4}dx-Ce^{-2r_{b}^{\frac{1}{2}}},\ \mbox{as}\ b\rightarrow0^+,
\endaligned
\end{equation}
and
\begin{equation}\label{eq2.15}
\int_{\mathbb R^2}V(x)\widehat{u}_{b}^{2}(x)dx=\frac{A_{b}^{2}}{a^\ast}\int_{\mathbb R^2}V(r_{b}^{-\frac{1}{2}}x+x_0)\xi^{2}(r_{b}^{-\frac{1}{2}}x)Q^{2}(x)dx=V(x_0)+o(1),
\end{equation}
where $o(1)\rightarrow0$ as $b\rightarrow0^+$. Similarly, we have
\begin{equation}\label{eq2.17}
\int_{\mathbb R^2}|\nabla\widehat{u}_{b}|^{2}dx
\leq\frac{r_{b}}{a^\ast}\int_{\mathbb R^2}|\nabla Q(x)|^{2}dx+Ce^{-r_{b}^{\frac{1}{2}}}
=\int_{\mathbb R^{2}}|\nabla\overline{u}_{b}|^{2}dx+Ce^{-r_{b}^{\frac{1}{2}}}\ \mbox{as}\ b\rightarrow0^+.
\end{equation}
Taking $x_{0}\in\mathbb R^{2}$ such that $V(x_{0})=0$, then the above estimates show that
\begin{equation*}
\aligned
0\leq e_{a}(b)-\overline{e}_{a}(b)
&\leq E_{a}^{b}(\widehat{u}_{b})-\overline{E}_{a}^{b}(\overline{u}_{b})=\overline{E}_{a}^{b}(\widehat{u}_{b})-\overline{E}_{a}^{b}(\overline{u}_{b})+\frac{1}{2}\int_{\mathbb R^{2}}V(x)|\widehat{u}_{b}|^{2}dx\\
&\leq\frac{1}{2}V(x_{0})+Ce^{-\frac{1}{2}r_{b}^{\frac{1}{2}}}+o(1)\rightarrow0,\ \mbox{as}\ b\rightarrow0^+,
\endaligned
\end{equation*}
and hence \eqref{eq2.10} holds. Moreover, since $u_{b}$ is a minimizer for $e_{a}(b)$, we know that
\begin{equation*}
\int_{\mathbb R^2}V(x)|u_{b}|^{2}dx=E_{a}^{b}(u_{b})-\overline{E}_{a}^{b}(u_{b})\leq e_{a}(b)-\overline{e}_{a}(b)\rightarrow0\ \mbox{as}\ b\rightarrow0^+.
\end{equation*}
This implies \eqref{eq2.11} holds and the proof of the lemma is completed.
\end{proof}
\begin{lemma}\label{lemma2.3}
For any given $a>a^{\ast}$, let $V(x)$ satisfy \eqref{eq1.4} and let $u_{b}$ be a nonnegative minimizer of $e_{a}(b)$. Then,
\begin{equation}\label{eq2.18}
\frac{\int_{\mathbb R^2}|\nabla u_{b}|^{2}dx}{r_{b}}\rightarrow1\quad \mbox{and}\quad \frac{\int_{\mathbb R^2}|u_{b}|^{4}dx}{r_{b}}\rightarrow\frac{2}{a^{\ast}}\ \mbox{as}\ b\rightarrow0^+,
\end{equation}
where $r_{b}$ is defined in \eqref{eq2.5}.
\end{lemma}
\begin{proof}[\textup{\textbf{Proof}}]
By contradiction, if there exits some $\theta\geq0$ and $\theta\neq1$ such that
\begin{equation*}
\frac{\int_{\mathbb R^2}|\nabla u_{b}|^{2}dx}{r_{b}}\rightarrow \theta\ \mbox{as}\ b\rightarrow0^+.
\end{equation*}
Then, there is always a contradiction for both $\theta \in [0,1)$ and $\theta >1$.

In fact, if $\theta\in[0,1)$, then there exists $\epsilon>0$ such that $\delta\triangleq\theta+\epsilon<1$ and
$\frac{\int_{\mathbb R^2}|\nabla u_{b}|^{2}dx}{r_{b}}\leq\delta$ as $b\rightarrow0^+$. It follows from \eqref{eq2.1}, \eqref{eq2.4} and \eqref{eq2.10} that
\begin{equation*}
0>e_{a}(b)=E_{a}^{b}(u_{b})\geq h\left(\int_{\mathbb R^2}|\nabla u_{b}|^{2}dx\right)\geq h(\delta r_{b})\geq h(r_{b})=\overline{e}_{a}(b)\ \mbox{as}\ b\ \mbox{close to 0},
\end{equation*}
where $h(\cdot)$ defined as \eqref{eq2.6} has a unique minimum point at $r_{b}$. Hence,
\begin{equation}\label{eq2.19*}
\lim_{b\rightarrow0^+}\frac{e_{a}(b)}{h(r_{b})}\leq
\lim_{b\rightarrow0^+}\frac{h(\delta r_{b})}{h(r_{b})}
=\lim_{b\rightarrow0^+}\frac{\frac{b}{4}\delta^{2}r_{b}^{2}-\frac{a-a^\ast}{2a^\ast}\delta r_{b}}{\frac{b}{4}r_{b}^{2}-\frac{a-a^\ast}{2a^\ast} r_{b}}
=-\delta^{2}+2\delta\in(0,1)\ \text{ for all }\ \delta\in[0,1).
\end{equation}
Moreover, \eqref{eq2.9*} and \eqref{eq2.10} indicate that
\begin{equation*}
\lim_{b\rightarrow0^+}\frac{e_{a}(b)}{h(r_{b})}=
\lim_{b\rightarrow0^+}\frac{\overline{e}_{a}(b)+o(1)}{\overline{e}_{a}(b)}
=1,
\end{equation*}
which contradict \eqref{eq2.19*}.

Similarly, if $\theta>1$, we have also a contradiction,
and hence $\frac{\int_{\mathbb R^2}|\nabla u_{b}|^{2}dx}{r_{b}}\rightarrow1$ as $b\rightarrow0^+$.

Since $u_{b}$ is a minimizer for $e_{a}(b)$, we have
\begin{equation*}
\frac{e_{a}(b)}{r_{b}}=\frac{\int_{\mathbb R^2}[|\nabla u_{b}|^{2}+V(x)u_{b}^{2}]dx}{2r_{b}}+\frac{b\left(\int_{\mathbb R^2}|\nabla u_{b}|^{2}dx\right)^{2}}{4r_{b}}-\frac{a\int_{\mathbb R^2}|u_{b}|^{4}dx}{4r_{b}}.
\end{equation*}
Applying \eqref{eq2.5}, $r_{b}\to\infty$ as $b\to0^{+}$ and $br_{b}=\frac{a-a^{\ast}}{a^{\ast}}$. It then follows from Lemmas \ref{lemma2.1} and \ref{lemma2.2} that
\begin{equation*}
\frac{a\int_{\mathbb R^2}|u_{b}|^{4}dx}{4r_{b}}\rightarrow\frac{a}{2a^{\ast}}\ \mbox{as}\ b\rightarrow0^+,
\end{equation*}
that is, $\frac{\int_{\mathbb R^2}|u_{b}|^{4}dx}{r_{b}}\rightarrow\frac{2}{a^{\ast}}$ as $b\rightarrow0^+$ and the lemma is proved.
\end{proof}
Our next lemma is to give the energy behaviors as $b\rightarrow0^+$ in the case of $a=a^\ast$.
\begin{lemma}\label{lemma2.4}
If $a=a^{\ast}$, let $V(x)$ satisfy \eqref{eq1.4} and let $u_{b}$ be a nonnegative minimizer of $e_{a^\ast}(b)$, then
\begin{equation}\label{eq2.19}
e_{a^\ast}(b)\rightarrow e_{a^\ast}(0)=0\ \mbox{and}\ \int_{\mathbb R^{2}}|\nabla u_{b}|^{2}dx\rightarrow+\infty,\ \mbox{as}\ b\rightarrow0^+.
\end{equation}
\begin{equation}\label{eq2.20}
b\left(\int_{\mathbb R^{2}}|\nabla u_{b}|^{2}dx\right)^{2}\rightarrow0\ \mbox{and}\ \int_{\mathbb R^{2}}V(x)|u_{b}|^{2}dx\rightarrow0,\ \mbox{as}\ b\rightarrow0^+.
\end{equation}
\end{lemma}
\begin{proof}[\textup{\textbf{Proof}}]
By Theorem 1 of \cite{GS}, $e_{a^\ast}(0)=0$. It then follows from \eqref{eq1.2} and \eqref{eq2.1} that
\begin{equation}\label{eq2.21}
e_{a^\ast}(b)=E_{a^\ast}^{b}(u_{b})
\geq\frac{b}{4}\left(\int_{\mathbb R^{2}}|\nabla u_{b}|^{2}dx\right)^{2}
+\frac{1}{2}\int_{\mathbb R^{2}}V(x)|u_{b}|^{2}dx>e_{a^\ast}(0)=0.
\end{equation}
Let $\xi(x)$ be the same cut-off function as \eqref{eq2.12}. For any $x_{0}\in\mathbb{R}^{2}$ and $\tau>0$, set
\begin{equation}\label{eq2.22}
u_{\tau}(x)=\frac{A_{\tau}\tau}{\sqrt{a^\ast}}\xi(x-x_{0})Q(\tau(x-x_{0})),
\end{equation}
where $A_{\tau}>0$ is chosen so that $\int_{\mathbb R^2}|u_{\tau}|^2dx=1$. Then, for $\tau$ large enough, similar to \eqref{eq2.14}--\eqref{eq2.17}, we have
\begin{equation}\label{eq2.23}
0\leq A_{\tau}^{2}-1\leq Ce^{-2\tau}\ \mbox{as}\ \tau\to\infty,
\end{equation}
\begin{equation}\label{eq2.24}
\int_{\mathbb R^{2}}|\nabla u_{\tau}|^{2}dx\leq\tau^{2}+Ce^{-\tau}\ \mbox{as}\ \tau\to\infty,
\end{equation}
\begin{equation}\label{eq2.25}
\int_{\mathbb R^{2}}|u_{\tau}|^{4}dx\geq\frac{2\tau^{2}}{a^{\ast}}-Ce^{-2\tau}\ \mbox{as}\ \tau\to\infty,
\end{equation}
and
\begin{equation}\label{eq2.26}
\int_{\mathbb R^{2}}V(x)|u_{\tau}|^{2}dx=V(x_{0})+o(1).
\end{equation}
where $o(1)\rightarrow0$ as $\tau\rightarrow+\infty$. Then, the above estimates show that
\begin{equation}\label{eq2.27}
0<e_{a^\ast}(b)\leq E_{a^\ast}^{b}(u_{\tau})
\leq\frac{b}{4}\tau^{4}+\frac{1}{2}V(x_{0})+Ce^{-\frac{1}{2}\tau}+o(1).
\end{equation}
Taking $x_{0}\in\mathbb R^{2}$ such that $V(x_{0})=0$, using \eqref{eq2.21}--\eqref{eq2.27} and letting $b\rightarrow0^{+}$ and $\tau\rightarrow+\infty$, we have
\begin{equation}\label{eq2.28}
0<e_{a^\ast}(b)\rightarrow0\ \mbox{as}\ b\rightarrow0^{+},
\end{equation}
and \eqref{eq2.21} implies that
\begin{equation}\label{eq2.29}
b\left(\int_{\mathbb R^{2}}|\nabla u_{b}|^{2}dx\right)^{2}\rightarrow0,\quad\int_{\mathbb R^{2}}V(x)|u_{b}|^{2}dx\rightarrow0\ \mbox{as}\ b\rightarrow0^{+}.
\end{equation}

Next, we claim that
\begin{equation}\label{eq2.30}
\int_{\mathbb R^{2}}|\nabla u_{b}|^{2}dx\rightarrow+\infty\ \mbox{as}\ b\rightarrow0^{+}.
\end{equation}
Otherwise, if \eqref{eq2.30} is false, then there exists a sequence of $\{b_{k}\}$ with $b_{k}\xrightarrow{k\to\infty}0^{+}$ such that the sequence $\{u_{k}\}$ is bounded in $\mathcal{H}$ since \eqref{eq2.29}, where $u_{k}\triangleq u_{b_k}$.
By the compact embedding results mentioned in Remark \ref{remark} , passing to a subsequence, there exists $u_{0}\in \mathcal{H}$ such that
\begin{equation}\label{eq2.31}
u_{k}\rightharpoonup u_{0}\ \mbox{in}\ \mathcal{H}\quad\mbox{and}\quad u_{k}\rightarrow u_{0}\ \mbox{in}\ L^{s}(\mathbb R^{2}),\  \mbox{as}\ k\rightarrow\infty,\ \mbox{for}\ s\in[2,+\infty).
\end{equation}
Then,
\begin{equation*}
e_{a^\ast}(0)\leq E_{a^\ast}^{0}(u_{0})\leq\liminf_{k\rightarrow\infty}E_{a^\ast}^{b_{k}}(u_{k})=\lim_{k\rightarrow\infty}e_{a^\ast}(b_{k})=0=e_{a^\ast}(0).
\end{equation*}
This means that $u_{0}$ is a minimizer of $e_{a^\ast}(0)$, which contradicts Theorem 1 of \cite{GS}, and hence \eqref{eq2.30} holds.
\end{proof}
\begin{lemma}\label{le2.3}
Suppose that $V(x)$ satisfies \eqref{eq1.4}. For any given $a\geq a^{\ast}$, let $u_{k}$ be a nonnegative minimizer of $e_{a}(b_{k})$ as in Theorem \ref{th1} and $z_{k}$ be a global maximum point of $u_{k}$, where $b_{k}\xrightarrow{k\to\infty}0^{+}$. Set
\begin{equation}\label{eq2.32}
w_{k}(x)=\epsilon_{k}u_{k}(\epsilon_{k}x+z_{k}),\ \mbox{where}\ \epsilon_{k}\ \mbox{is defined by \eqref{eq1.9}}.
\end{equation}
Then,
\begin{equation}\label{eq2.33}
\liminf_{k\rightarrow\infty}\int_{B_{2}(0)}|w_{k}|^{2}dx\geq M>0\ \ \mbox{for some}\ M>0.
\end{equation}
Moreover, passing to a subsequence, there exists a $z_{0}\in\mathbb R^{2}$ such that
\begin{equation}\label{eq2.34}
z_{k}\rightarrow z_{0}\  \mbox{as}\ k\rightarrow\infty,\ \mbox{and}\ V(z_{0})=0.
\end{equation}
\end{lemma}
\begin{proof}[\textup{\textbf{Proof}}]
Since $u_{k}$ is a nonnegative minimizer for $e_{a}(b_{k})$. Then, $u_{k}(x)$ satisfies the following Euler-Lagrange equation
\begin{equation}\label{eq2.35}
-\left(1+b_{k}\int_{\mathbb R^{2}}|\nabla u_{k}|^{2}dx\right)\Delta u_{k}+V(x)u_{k}=\mu_{k} u_{k}+au^{3}_{k},\  x\in \mathbb{R}^{2},
\end{equation}
where $\mu_{k}\in\mathbb R$ is a suitable Lagrange multiplier associated to $u_{k}$, and
\begin{equation}\label{eq2.36}
e_{a}(b_{k})=\frac{1}{2}\int_{\mathbb R^{2}}[|\nabla u_{k}|^{2}+V(x)|u_{k}|^{2}]dx+\frac{b_{k}}{4}\left(\int_{\mathbb R^{2}}|\nabla u_{k}|^{2}dx\right)^{2}-\frac{a}{4}\int_{\mathbb R^{2}}|u_{k}|^{4}dx.
\end{equation}
Moreover,
\begin{equation}\label{eq2.37}
\mu_{k}=\int_{\mathbb R^{2}}[|\nabla u_{k}|^{2}+V(x)|u_{k}|^{2}]dx+b_{k}\left(\int_{\mathbb R^{2}}|\nabla u_{k}|^{2}dx\right)^{2}-a\int_{\mathbb R^{2}}|u_{k}|^{4}dx.
\end{equation}
If $a=a^{\ast}$, we deduce from \eqref{eq1.9}, \eqref{eq2.36} and Lemma \ref{lemma2.4} that
\begin{equation}\label{eq2.38}
\begin{cases}
\int_{\mathbb R^{2}}|\nabla w_{k}|^{2}dx=\int_{\mathbb R^{2}}|w_{k}|^{2}dx=1,\\\\
\int_{\mathbb R^{2}}|w_{k}|^{4}dx=\epsilon_{k}^{2}\int_{\mathbb R^{2}}|u_{k}|^{4}dx\rightarrow\frac{2}{a^{\ast}}\ \mbox{as}\ k\rightarrow\infty.
\end{cases}
\end{equation}
If $a>a^{\ast}$, it follows from \eqref{eq1.9}, \eqref{eq2.32} and Lemma \ref{lemma2.3} that
\begin{equation}\label{eq2.39}
\begin{cases}
\int_{\mathbb R^{2}}|w_{k}|^{2}dx=1,  \ \epsilon_{k}^{2}= {r_{b_{k}}}^{-1}, \\\\
\int_{\mathbb R^{2}}|\nabla w_{k}|^{2}dx=\epsilon_{k}^{2}\int_{\mathbb R^{2}}|\nabla u_{k}|^{2}dx=\frac{\int_{\mathbb R^2}|\nabla u_{k}|^{2}dx}{r_{b_{k}}}\rightarrow1\ \mbox{as}\ k\rightarrow\infty,\\\\
\int_{\mathbb R^{2}}|w_{k}|^{4}dx=\epsilon_{k}^{2}\int_{\mathbb R^{2}}|u_{k}|^{4}dx=\frac{\int_{\mathbb R^2}|u_{k}|^{4}dx}{r_{b_{k}}}\rightarrow\frac{2}{a^{\ast}}\ \mbox{as}\ k\rightarrow\infty.
\end{cases}
\end{equation}
Hence, for any given $a\geq a^\ast$, combining \eqref{eq2.37} and the above estimates, we see that
\begin{equation}\label{eq2.40}
\mu_{k}\epsilon_{k}^{2}\rightarrow-\frac{a}{a^{\ast}}\ \mbox{as}\ k\rightarrow\infty.
\end{equation}
Since $u_{k}$ satisfies \eqref{eq2.35}, by the definition of $w_{k}$ in \eqref{eq2.32}, we know that
$w_{k}(x)$ satisfies
\begin{equation}\label{eq2.41}
-\left(1+b_{k}\int_{\mathbb R^{2}}|\nabla u_{k}|^{2}dx\right)\Delta w_{k}+\epsilon_{k}^{2} V(\epsilon_{k}x+z_{k})w_{k}(x)
=\mu_{k}\epsilon_{k}^{2}w_{k}(x)
+aw_{k}^{3}(x),\  x\in\mathbb R^{2}.
\end{equation}
Hence, as $k$ large enough, it follows from \eqref{eq2.40} that
\begin{equation}\label{eq2.42}
-\Delta w_{k}-c(x)w_{k}\leq0,\ \mbox{where}\ c(x)=aw_{k}^{2}(x).
\end{equation}
Applying De Giorgi-Nash-Moser theory(similar to the proof of \cite[Theorem 4.1]{HL2}), we deduce that
\begin{equation}\label{eq2.43}
\max_{B_{1}(\xi)}w_{k}(x)\leq C\left(\int_{B_{2}(\xi)}|w_{k}|^{2}dx\right)^{\frac{1}{2}},
\end{equation}
where $\xi$ is an arbitrary point in $\mathbb R^{2}$ and $C$ is a constant depending only on the bound of $\|w_{k}\|_{L^{4}(B_{2}(\xi))}$.
Since $z_{k}$ is a global maximum point of $u_{k}$, 0 is a global maximum point of $w_{k}$. We claim that there exists some $\eta>0$ such that
\begin{equation}\label{eq2.44}
w_{k}(0)\geq\eta\quad\mbox{for}\ k\ \mbox{large enough}.
\end{equation}
If \eqref{eq2.44} is false, then for any $r>0$, passing to a subsequence if necessary, we have
\begin{equation*}
\sup_{y\in\mathbb R^{2}}\int_{B(y,r)}|w_{k}|^{2}(x)dx\rightarrow0\ \mbox{as}\ k\rightarrow\infty.
\end{equation*}
Then, the vanishing Lemma 1.21 in \cite{WM} shows that $\int_{\mathbb R^{2}}|w_{k}|^{4}dx\rightarrow0$ as $k\rightarrow\infty$, which contradicts \eqref{eq2.38} and \eqref{eq2.39}. Hence, \eqref{eq2.44} holds, and \eqref{eq2.33} follows from \eqref{eq2.43} and \eqref{eq2.44}.

Next, using \eqref{eq2.11} and \eqref{eq2.20}, we know that
\begin{equation*}
0=\liminf_{k\rightarrow\infty}\int_{\mathbb R^{2}}V(\epsilon_{k}x+z_{k})|w_{k}|^{2}dx
\geq\liminf_{k\rightarrow\infty}\int_{B_{2}(0)}V(\epsilon_{k}x+z_{k})|w_{k}|^{2}dx.
\end{equation*}
Since $V(x)\to\infty$ as $|x|\rightarrow\infty$, \eqref{eq2.33} implies that $\{z_{k}\}$ is a bounded sequence in $\mathbb R^2$, and passing to a subsequence if necessary, there exists a $z_{0}\in\mathbb R^2$ such that \eqref{eq2.34} holds
\end{proof}
Now, we are ready to prove Theorem \ref{th1}.
\begin{proof}[\textup{\textbf{Proof of Theorem \ref{th1}}}]
Let $u_{k}$ be a nonnegative minimizer of $e_{a}(b_{k})$ for $a\geq a^{\ast}$ and $w_{k}(x)$ be defined in \eqref{eq2.32}. It follows from \eqref{eq2.38} and \eqref{eq2.39} that $\{w_{k}\}$ is a bounded sequence in $H^{1}(\mathbb R^{2})$, and passing to subsequence, there exists $w_{0}\in H^{1}(\mathbb R^{2})$ such that
\begin{equation}\label{eq2.48*}
w_{k}\rightharpoonup w_{0}\geq0\ \mbox{in}\ H^{1}(\mathbb R^{2})\ \mbox{as}\ k\rightarrow\infty.
\end{equation}
Moreover, since $w_{k}(x)$ satisfies \eqref{eq2.41},
applying \eqref{eq2.40} and passing to the weak limit, we know that $w_{0}(x)$ satisfies, in the weak sense,
\begin{equation}\label{eq2.48}
-\Delta w_{0}+w_{0}(x)=a^{\ast}w_{0}^{3}(x),\  x\in\mathbb R^{2}.
\end{equation}
Furthermore, \eqref{eq2.33} implies that $w_{0}(x)\not\equiv0$, and $w_{0}(x)>0$ since the strong maximum principle. Comparing \eqref{eq1.5} and \eqref{eq2.48}, the uniqueness of positive solution of \eqref{eq1.5} shows that
\begin{equation}\label{eq2.49}
w_{0}(x)=\frac{Q(|x-x_{0}|)}{\sqrt{a^\ast}}\  \mbox{for}\ \mbox{some}\ x_{0}\in\mathbb R^{2},
\end{equation}
where $Q(x)$ is the unique positive solution of \eqref{eq1.5}. Moreover, by \eqref{eq2.1} we have
\begin{equation*}
\int_{\mathbb R^2}|\nabla w_{0}|^{2}dx=\int_{\mathbb R^2}w_{0}^{2}dx=1.
\end{equation*}
It then follows from \eqref{eq2.39}, \eqref{eq2.40} and \eqref{eq2.48*} that
\begin{equation}\label{eq2.50}
w_{k}\rightarrow w_{0}=\frac{Q(|x-x_{0}|)}{\sqrt{a^\ast}}\  \mbox{in}\ H^{1}(\mathbb R^{2})\ \mbox{as}\ k\rightarrow\infty.
\end{equation}
Since $V(x)\in C_{loc}^{\alpha}(\mathbb R^{2},\mathbb R^{+})$ for some $\alpha\in(0,1)$. Similar to the Proof of Theorem 1.2 in \cite{GZZ3}, we know from \eqref{eq2.41} and \eqref{eq2.50} that
\begin{equation}\label{eq2.51}
w_{k}\rightarrow w_{0}\ \mbox{in}\ C_{loc}^{2,\alpha}(\mathbb R^{2})\ \mbox{as}\ k\rightarrow\infty\ \mbox{for some}\ \alpha\in(0,1).
\end{equation}
By \eqref{eq2.32}, $x=0$ is a critical (global maximum) point of $w_{k}(x)$ for all $k>0$, it is also a critical point of $w_{0}$ by \eqref{eq2.51}. Since $Q(x)$ is radially symmetric about the origin and strictly monotonous about $|x|$ (see e.g., \cite{GNN,KJ,W2}), then $w_{0}(x)$ has a unique global maximum point $x=0$ and $x_{0}=0$. Hence,
\begin{equation*}
w_{0}(x)=\frac{Q(|x|)}{\sqrt{a^\ast}}.
\end{equation*}
Moreover, using \eqref{eq2.51}, similar to the proof of \cite[Theorem 1.1]{GZZ}, we deduce that $z_k$ is the unique maximum point of $u_k$ and $z_{k}$ goes to a global minimum point of $V(x)$ as $k\rightarrow\infty$ by \eqref{eq2.34}.
\end{proof}
\section{Concentration behavior for homogeneous type potential.}
\quad The aim of this section is to show that, if there are more information on the global minimum point of $V(x)$, such as \eqref{eq1.13} and \eqref{eq1.14}, then we can get more precise description on the concentration behavior for the minimizers of \eqref{eq1.2} as $b\rightarrow0^+$, i.e., Theorem \ref{th2}. To prove this Theorem, we need some detailed estimates on the energy $e_{a}(b)$ for $a=a^\ast$ as $b\rightarrow0^+$.
\begin{lemma}\label{lemma3.1}
Let $V(x)$ satisfy \eqref{eq1.4} and \eqref{eq1.13}--\eqref{eq1.14}. If $a=a^\ast$, then
\begin{equation}\label{eq3.1}
\limsup_{b\rightarrow0^+}\frac{e_{a^\ast}(b)}{b^{\frac{p}{p+4}}}\leq\frac{p+4}{4p}\left(\frac{p\lambda_{0}}{2a^\ast}\right)^{\frac{4}{p+4}},
\end{equation}
where $p$ and $\lambda_{0}$ are given by \eqref{eq1.16} and \eqref{eq1.17}, respectively.
\end{lemma}
\begin{proof}[\textup{\textbf{Proof}}]
Take $x_{i_{0}}\in Z_{0}$ and $y_0$ satisfying $H_{i_{0}}(y_{0})=\lambda_{0}$. Let $u_{\tau}(x)$ be given by \eqref{eq2.22} and take $x_{0}=x_{i_{0}}+\frac{1}{\tau}y_{0}$.
Then, it follows from \eqref{eq1.14} and \eqref{eq2.23}--\eqref{eq2.25} that
\begin{equation}\label{eq3.2}
\aligned
\int_{\mathbb R^{2}}V(x)|u_{\tau}|^{2}dx&=
\frac{A_{\tau}^{2}}{\|Q\|_{L^{2}}^{2}}\int_{\mathbb R^{2}}V((x+y_{0})/\tau+x_{i_{0}})\xi^{2}(x/\tau)Q^{2}(x)dx\\
&=\frac{A_{\tau}^{2}}{\|Q\|_{L^{2}}^{2}}\int_{\mathbb R^{2}}\frac{V((x+y_{0})/\tau+x_{i_{0}})}{V_{i_{0}}((x+y_{0})/\tau)}V_{i_{0}}((x+y_{0})/\tau)\xi^{2}(x/\tau)Q^{2}(x)dx\\
&=\frac{\lambda_{0}}{a^\ast\tau^{p}}(1+o(1))\ \mbox{as}\ \tau\rightarrow\infty,
\endaligned
\end{equation}
and
\begin{equation}\label{eq3.3}
e_{a^\ast}(b)\leq E_{a^\ast}^{b}(u_{\tau})=\frac{b}{4}\tau^{4}+\frac{\lambda_{0}}{2a^\ast\tau^{p}}(1+o(1))+Ce^{-\frac{1}{2}\tau}\ \mbox{as}\ \tau\rightarrow\infty.
\end{equation}
Take $\tau=\left(\frac{p\lambda_{0}}{2ba^\ast}\right)^{\frac{1}{p+4}}$, then $\tau\rightarrow\infty\ \mbox{as}\ b\rightarrow0^+$. It follows from \eqref{eq3.2} and \eqref{eq3.3} that
\begin{equation*}
e_{a^\ast}(b)\leq \frac{b^{\frac{p}{p+4}}}{4}\left(\frac{p\lambda_{0}}{2a^\ast}\right)^{\frac{4}{p+4}}+
\frac{b^{\frac{p}{p+4}}}{p}\left(\frac{p\lambda_{0}}{2a^\ast}\right)^{\frac{4}{p+4}}(1+o(1))+Ce^{-\frac{1}{2}\tau}\ \mbox{as}\ b\rightarrow0^+.
\end{equation*}
This shows that
\begin{equation*}
\limsup_{b\rightarrow0^+}\frac{e_{a^\ast}(b)}{b^{\frac{p}{p+4}}}\leq\frac{1}{4}\left(\frac{p\lambda_{0}}{2a^\ast}\right)^{\frac{4}{p+4}}+
\frac{1}{p}\left(\frac{p\lambda_{0}}{2a^\ast}\right)^{\frac{4}{p+4}}=\frac{p+4}{4p}\left(\frac{p\lambda_{0}}{2a^\ast}\right)^{\frac{4}{p+4}}.
\end{equation*}
\end{proof}
Now, we come to prove Theorem \ref{th2}.
\begin{proof}[\textup{\textbf{Proof of Theorem \ref{th2}}}]
We prove this theorem by considering two cases, respectively.

\textbf{Case I : $a=a^\ast$.} By Lemma \ref{lemma3.1}, we know that
$\underset{k\to\infty}\limsup\frac{e_{a^\ast}(b_{k})}{b_{k}^{\frac{p}{p+4}}}$
has a upper estimates. Therefore, we need only to show that the limit has the same lower bound to prove \eqref{eq1.21*} of Theorem \ref{th2}.
Let $u_{k}$ be a nonnegative minimizer for $e_{a^\ast}(b_{k})$ and $w_{k}(x)$ be defined by \eqref{eq2.32}, where $b_{k}\xrightarrow{k\to\infty}0^{+}$. Then, passing to a subsequence, we know from  Theorem \ref{th1} that each $u_{k}$ has a
unique maximum point $z_{k}$ such that $z_{k}\rightarrow x_{0} \mbox{ as }\ k\rightarrow\infty\ \mbox{with}\ V(x_0)=0$. We may assume $x_{0}=x_{i_0}$ for some $1\leq i_{0}\leq m$.

We claim that
\begin{equation}\label{eq3.4}
p_{i_{0}}=p\ \ \mbox{and }\ \left\{\frac{z_{k}-x_{i_{0}}}{\epsilon_{k}}\right\}\mbox{ is bounded},
\end{equation}
where $\epsilon_{k}$ is given by \eqref{eq1.9}. Otherwise, if $p_{i_{0}}<p$ or $\underset{k\to\infty}\lim\left|\frac{z_{k}-x_{i_{0}}}{\epsilon_{k}}\right|=+\infty$,  then using $V_{i_{0}}(tx)=t^{p_{i_{0}}}V_{i_{0}}(tx)$, \eqref{eq1.8} and \eqref{eq1.14} imply that, for any $M>0$ large enough,
\begin{equation*}
\aligned
\liminf_{k\rightarrow\infty}\frac{1}{\epsilon_{k}^{p}}\int_{\mathbb R^{2}}V(x)|u_{k}|^{2}dx
&=\liminf_{k\rightarrow\infty}\frac{1}{\epsilon_{k}^{p}}\int_{\mathbb R^{2}}V(\epsilon_{k}x+z_{k})|w_{k}|^{2}dx\\
&=\liminf_{k\rightarrow\infty}\frac{1}{\epsilon_{k}^{p-p_{i_{0}}}}\int_{\mathbb R^{2}}\frac{V(\epsilon_{k}x+z_{k})}{V_{i_{0}}
(\epsilon_{k}x+z_{k}-x_{i_{0}})}V_{i_{0}}\left(x+\frac{z_{k}-x_{i_{0}}}{\epsilon_{k}}\right)|w_{k}|^{2}dx\geq M.
\endaligned
\end{equation*}
Hence, by the Gagliardo-Nirenberg inequality \eqref{eq2.1} and Young's inequality, we see that
\begin{equation}\label{eq3.5}
e_{a^\ast}(b_k)=E_{a^\ast}^{b_k}(u_{k})\geq\frac{b_{k}}{4\epsilon_{k}^{4}}+\frac{M}{2}\epsilon_{k}^{p}\geq CM^{\frac{4}{p+4}}b_{k}^{\frac{p}{p+4}},
\end{equation}
which contradicts \eqref{eq3.1} if $M>0$ large enough. So, \eqref{eq3.4} is proved. Therefore, passing to a subsequence, we may assume that there exists  $y_{0}\in\mathbb R^2$ such that
\begin{equation}\label{eq3.5*}
\frac{z_{k}-x_{i_{0}}}{\epsilon_{k}}\rightarrow y_{0}\ \mbox{as}\ k\rightarrow\infty.
\end{equation}
It follows from \eqref{eq1.8} and \eqref{eq1.14} that
\begin{equation}\label{eq3.6}
\aligned
\liminf_{k\rightarrow\infty}\frac{1}{\epsilon_{k}^{p}}\int_{\mathbb R^{2}}V(x)|u_{k}|^{2}dx
&=\liminf_{k\rightarrow\infty}\frac{1}{\epsilon_{k}^{p}}\int_{\mathbb R^{2}}V(\epsilon_{k}x+z_{k})|w_{k}|^{2}dx\\
&=\liminf_{k\rightarrow\infty}\int_{\mathbb R^{2}}\frac{V(\epsilon_{k}x+z_{k})}{V_{i_{0}}(\epsilon_{k}x+z_{k}-x_{i_{0}})}V_{i_{0}}
\left(x+\frac{z_{k}-x_{i_{0}}}{\epsilon_{k}}\right)|w_{k}|^{2}dx\\
&=\frac{1}{a^\ast}\int_{\mathbb R^{2}}V_{i_{0}}(x+y_{0})Q^2(x)dx\\
&\geq\frac{\lambda_{i_{0}}}{a^\ast}\geq\frac{\lambda_{0}}{a^\ast}.
\endaligned
\end{equation}
This and Young's inequality imply that
\begin{equation*}
e_{a^\ast}(b_k)=E_{a^\ast}^{b_k}(u_{k})\geq\frac{b_{k}}{4\epsilon_{k}^{4}}+\frac{\lambda_{0}\epsilon_{k}^{p}}{2a^\ast}(1+o(1))
\geq\frac{b_{k}^{\frac{p}{p+4}}(p+4)(1+o(1))}{4p}\left(\frac{p\lambda_{0}}{2a^\ast}\right)^{\frac{4}{p+4}}.
\end{equation*}
Hence,
\begin{equation}\label{eq3.7}
\liminf_{k\rightarrow\infty}\frac{e_{a^\ast}(b_{k})}{b_{k}^{\frac{p}{p+4}}}\geq\frac{p+4}{4p}\left(\frac{p\lambda_{0}}{2a^\ast}\right)^{\frac{4}{p+4}},
\end{equation}
where the equality holds if and only if $\lambda_{i_{0}}=\lambda_{0}$, $H_{i_{0}}(y_{0})=\lambda_{0}$ and
\begin{equation}\label{eq3.8}
\lim_{k\rightarrow\infty}\frac{\epsilon_{k}}{\overline{\epsilon}_{k}}=1,\ \mbox{where}\ \overline{\epsilon}_{k}\ \mbox{defined in}\ \eqref{eq1.20}.
\end{equation}
Moreover, it follows from Lemma \ref{lemma3.1} that
\begin{equation}\label{eq3.7*}
\lim_{k\rightarrow\infty}\frac{e_{a^\ast}(b_{k})}{b_{k}^{\frac{p}{p+4}}}=\frac{p+4}{4p}\left(\frac{p\lambda_{0}}{2a^\ast}\right)^{\frac{4}{p+4}}.
\end{equation}
This shows that all inequalities in \eqref{eq3.6} and \eqref{eq3.7} become equalities, and $H_{i_{0}}(y_{0})=\lambda_{i_0}=\lambda_{0}$. Therefore, \eqref{eq1.19} follows from \eqref{eq1.8} and \eqref{eq3.8}. Also \eqref{eq1.21**} follows from \eqref{eq3.5*}.

\textbf{Case II : $a>a^\ast$.} Take $x_{i_{0}}\in Z_{0}$ and $y_0$ satisfying $H_{i_{0}}(y_{0})=\lambda_{0}$. Let
\begin{equation*}
\widehat{u}_{k}(x)=A_{k}\xi(x-x_{i_{0}}-\overline{\epsilon}_{k}y_0)\overline{u}_{k}(x-x_{i_{0}}-\overline{\epsilon}_{k}y_0),
\end{equation*}
where $\overline{u}_{k}(x)\triangleq \overline{u}_{b_{k}}(x)$ and $\xi(x)$ are defined in \eqref{eq2.5} and  \eqref{eq2.12}, respectively, and $A_{k}>0$ is chosen so that $\int_{\mathbb R^2}|\widehat{u}_{k}|^2dx=1$. Note from \eqref{eq1.20} and \eqref{eq2.5} that $\overline{\epsilon}_{k}=r_{b_{k}}^{-\frac{1}{2}}$. Then, similar to \eqref{eq2.14}--\eqref{eq2.17}, it follows \eqref{eq1.14} and the dominated convergence theorem that
\begin{equation*}
\aligned
e_{a}(b_{k})-\overline{e}_{a}(b_{k})
&\leq E_{a}^{b_{k}}(\widehat{u}_{k})-\overline{E}_{a}^{b_{k}}(\overline{u}_{k})
\leq\frac{1}{2}\int_{\mathbb R^{2}}V(x)|\widehat{u}_{k}|^{2}dx+O(e^{-\frac{1}{2\overline{\epsilon}_{k}}})\\
&=\frac{1}{2a^\ast}\int_{\mathbb R^{2}}V(\overline{\epsilon}_{k}x+x_{i_{0}}+\overline{\epsilon}_{k}y_{0})\xi^{2}(\overline{\epsilon}_{k}x)Q^{2}(x)dx
+O(e^{-\frac{1}{2\overline{\epsilon}_{k}}})\\
&=\frac{1}{2a^\ast}\int_{\mathbb R^{2}}\frac{V(\overline{\epsilon}_{k}x+x_{i_{0}}+\overline{\epsilon}_{k}y_{0})}{V_{i_{0}}(\overline{\epsilon}_{k}x+\overline{\epsilon}_{k}y_{0})}
V_{i_{0}}(\overline{\epsilon}_{k}x+\overline{\epsilon}_{k}y_{0})
\xi^{2}(\overline{\epsilon}_{k}x)Q^{2}(x)dx
+O(e^{-\frac{1}{2\overline{\epsilon}_{k}}})\\
&=\frac{\overline{\epsilon}_{k}^{p}}{2a^\ast}\int_{\mathbb R^{2}}\frac{V(\overline{\epsilon}_{k}x+x_{i_{0}}+\overline{\epsilon}_{k}y_{0})}{V_{i_{0}}(\overline{\epsilon}_{k}x+\overline{\epsilon}_{k}y_{0})}
V_{i_{0}}(x+y_{0})
\xi^{2}(\overline{\epsilon}_{k}x)Q^{2}(x)dx
+O(e^{-\frac{1}{2\overline{\epsilon}_{k}}})\\
&=\frac{\overline{\epsilon}_{k}^{p}\lambda_{0}}{2a^\ast}(1+o(1))\ \mbox{as}\ k\rightarrow\infty.
\endaligned
\end{equation*}
This implies that
\begin{equation}\label{eq3.9}
\limsup_{k\rightarrow\infty}\frac{e_{a}(b_{k})-\overline{e}_{a}(b_{k})}{\overline{\epsilon}_{k}^{p}}\leq\frac{\lambda_{0}}{2a^\ast}.
\end{equation}

Let $u_{k}$ be a nonnegative minimizer of $e_{a}(b_{k})$, passing to a subsequence, we know from  Theorem \ref{th1} that each $u_{k}$ has a
unique maximum point $z_{k}$ such that $z_{k}\rightarrow x_{0} \mbox{ as }\ k\rightarrow\infty\ \mbox{with}\ V(x_0)=0$. We may assume $x_{0}=x_{i_0}$ for some $1\leq i_{0}\leq m$.
It follows from \eqref{eq1.8} and \eqref{eq1.14} that
\begin{equation}\label{eq3.99*}
\aligned
&\liminf_{k\rightarrow\infty}\frac{e_{a}(b_{k})-\overline{e}_{a}(b_{k})}{\overline{\epsilon}_{k}^{p}}
\geq\liminf_{k\rightarrow\infty}\frac{1}{2\overline{\epsilon}_{k}^{p}}\int_{\mathbb R^{2}}V(x)|u_{k}|^{2}dx\\
&\quad=\liminf_{k\rightarrow\infty}\frac{1}{2\overline{\epsilon}_{k}^{p}}\int_{\mathbb R^{2}}V(\overline{\epsilon}_{k}x+z_{k})|w_{k}(x)|^{2}dx\\
&\quad=\liminf_{k\rightarrow\infty}\frac{1}{2\overline{\epsilon}_{k}^{p-p_{i_0}}}\int_{\mathbb R^{2}}\frac{V(\overline{\epsilon}_{k}x+z_{k}-x_{i_{0}}+x_{i_{0}})}{V_{i_{0}}(\overline{\epsilon}_{k}x+z_{k}-x_{i_{0}})}
V_{i_{0}}\left(x+\frac{z_{k}-x_{i_{0}}}{\overline{\epsilon}_{k}}\right)|w_{k}(x)|^{2}dx.
\endaligned
\end{equation}
Combining \eqref{eq3.9} and \eqref{eq3.99*}, we deduce from \eqref{eq1.4} and \eqref{eq1.14} that $p_{i_{0}}=p$ and $\left\{\frac{z_{k}-x_{i_{0}}}{\overline{\epsilon}_{k}}\right\}$ is bounded sequence in $\mathbb R^2$. So, we may assume that there exists $y_0\in\mathbb R^2$ such that
\begin{equation}\label{eq3.10*}
\frac{z_{k}-x_{i_{0}}}{\overline{\epsilon}_{k}}\rightarrow y_{0}\ \mbox{as}\ k\rightarrow\infty.
\end{equation}
Using \eqref{eq1.8} and \eqref{eq1.14} we know that
\begin{equation}\label{eq3.10}
\liminf_{k\rightarrow\infty}\frac{e_{a}(b_{k})-\overline{e}_{a}(b_{k})}{\overline{\epsilon}_{k}^{p}}
\geq\frac{1}{2a^\ast}\int_{\mathbb R^{2}}V_{i_{0}}(x+y_{0})Q^{2}(x)dx
\geq\frac{\lambda_{i_{0}}}{2a^\ast}
\geq\frac{\lambda_{0}}{2a^\ast},
\end{equation}
where \eqref{eq3.10} becomes equalities if and only if $\lambda_{i_{0}}=\lambda_{0}$ and $H_{i_{0}}(y_{0})=\lambda_{0}$. Combining \eqref{eq3.9} and \eqref{eq3.10}, we have
\begin{equation*}
\lim_{k\rightarrow\infty}\frac{e_{a}(b_{k})-\overline{e}_{a}(b_{k})}{\overline{\epsilon}_{k}^{p}}=\frac{\lambda_{0}}{2a^\ast}.
\end{equation*}
Therefore, \eqref{eq3.10} becomes equalities, and hence \eqref{eq1.21**} holds by \eqref{eq3.10*}.
\end{proof}
\section{Uniqueness of minimizers: Proof of Theorem \ref{th4}.}
\quad In this section, we come to prove Theorem \ref{th4}, that is, the uniqueness of the minimizers of $e_{a}(b)$ when $b>0$ small enough. For this purpose, we argue by contradiction. If for any given $a\geq a^\ast$, there exist two different nonnegative minimizers $u_{1,k}$ and $u_{2,k}$ for $e_{a}(b_{k})$. Let $z_{1,k}$ and $z_{2,k}$ be the unique maximum point of $u_{1,k}$ and $u_{2,k}$, respectively. By \eqref{eq2.35}, the minimizers $u_{i,k}$ satisfy the following equation
\begin{equation}\label{eq4.1}
-\left(1+b_{k}\int_{\mathbb R^{2}}|\nabla u_{i,k}|^{2}dx\right)\Delta u_{i,k}+V(x)u_{i,k}=\mu_{i,k} u_{i,k}+au_{i,k}^{3},\  x\in \mathbb{R}^{2},\ i=1,2,
\end{equation}
where $\mu_{i,k}\in\mathbb R$ is a suitable lagrange multiplier.
Since $u_{1,k}(x)\not\equiv u_{2,k}(x)$, set
\begin{equation}\label{eq4.2}
\overline{u}_{i,k}(x)=\sqrt{a^\ast}\overline{\epsilon}_{k}u_{i,k}(\overline{\epsilon}_{k}x+z_{1,k}),\ i=1,2,
\end{equation}
and
\begin{equation}\label{eq4.3}
\overline{\eta}_{k}(x)=\frac{\overline{u}_{1,k}(x)-\overline{u}_{2,k}(x)}{||\overline{u}_{1,k}-\overline{u}_{2,k}||_{L^{\infty}}}.
\end{equation}
It follows from Theorem \ref{th2} that
\begin{equation*}
\overline{u}_{i,k}(x)\rightarrow Q(x)\ \mbox{uniformly in}\ \mathbb R^{2}\ \mbox{as}\ k\rightarrow\infty.
\end{equation*}
Moreover, $\overline{u}_{i,k}$ and $\overline{\eta}_{k}$ satisfy
\begin{equation}\label{eq4.4}
-\left(1+\frac{b_{k}}{a^{\ast}\overline{\epsilon}_{k}^{2}}\int_{\mathbb R^{2}}|\nabla \overline{u}_{i,k}|^{2}dx\right)\Delta \overline{u}_{i,k}+\overline{\epsilon}_{k}^{2}V(\overline{\epsilon}_{k}x+z_{1,k})\overline{u}_{i,k}
=\mu_{i,k}\overline{\epsilon}_{k}^{2} \overline{u}_{i,k}+\frac{a}{a^\ast}\overline{u}_{i,k}^{3},\  x\in \mathbb{R}^{2},\ i=1,2,
\end{equation}
and
\begin{equation}\label{eq4.5}
\aligned
-\left(1+\frac{b_{k}}{a^{\ast}\overline{\epsilon}_{k}^{2}}\int_{\mathbb R^{2}}|\nabla \overline{u}_{2,k}|^{2}dx\right)\Delta \overline{\eta}_{k}
&-\frac{b_{k}}{a^{\ast}\overline{\epsilon}_{k}^{2}}\int_{\mathbb R^{2}}\nabla(\overline{u}_{1,k}
+\overline{u}_{2,k})\nabla\overline{\eta}_{k}dx\Delta\overline{u}_{1,k}\\
&+\overline{\epsilon}_{k}^{2}V(\overline{\epsilon}_{k}x+z_{1,k})\overline{\eta}_{k}
=\mu_{1,k}\overline{\epsilon}_{k}^{2} \overline{\eta}_{k}+\overline{g}_{k}(x)+\overline{f}_{k}(x),
\endaligned
\end{equation}
where
\begin{equation}\label{eq4.6}
\overline{g}_{k}(x)=\overline{\epsilon}_{k}^{2}\frac{\mu_{1,k}-\mu_{2,k}}{||\overline{u}_{1,k}-\overline{u}_{2,k}||_{L^{\infty}}}\overline{u}_{2,k}\ \mbox{and}\ \overline{f}_{k}(x)=\frac{a\overline{\eta}_{k}}{a^\ast}(\overline{u}_{1,k}^{2}+\overline{u}_{1,k}\overline{u}_{2,k}+\overline{u}_{2,k}^{2}).
\end{equation}

The following lemma gives the decay estimates on $\overline{u}_{i,k}$ and $|\nabla \overline{u}_{i,k}|$ ($i=1,2$), which are required in proving Theorem \ref{th4}.
\begin{lemma}\label{lemma4.1}
Let $a\geq a^\ast$ and $V(x)$ satisfy the assumptions of Theorem \ref{th4}. If $u_{i,k}$ $(i=1,2)$ are two nonnegative minimizers of $e_{a}(b_k)$ with $b_{k}\xrightarrow{k\to\infty}0^{+}$. Then, up to a subsequence, we have
\begin{equation}\label{eq4.7}
\overline{u}_{i,k}(x)\leq Ce^{-\frac{|x|}{2}}\ \mbox{and}\ \ |\nabla \overline{u}_{i,k}(x)|\leq Ce^{-\frac{|x|}{4}}\ \mbox{as}\ |x|\rightarrow\infty,\ i=1,2,
\end{equation}
where $C>0$ is a constant independent of $k$.
\end{lemma}
\begin{proof}[\textup{\textbf{Proof}}]
By \eqref{eq1.9}, \eqref{eq2.40} and \eqref{eq3.8}, we know that
\begin{equation}\label{eq4.8}
\mu_{i,k}\overline{\epsilon}_{k}^{2}\rightarrow-\frac{a}{a^\ast}\ \mbox{as}\ k\rightarrow\infty.
\end{equation}
From \eqref{eq4.4}, one can derive that
\begin{equation*}
-\Delta \overline{u}_{i,k}\leq c_{i,k}(x)\overline{u}_{i,k}\ \mbox{in}\ \mathbb R^2,\ \mbox{where}\ c_{i,k}(x)=\frac{a}{a^{\ast}}\overline{u}_{i,k}^{2}.
\end{equation*}
Applying De Giorgi-Nash-Moser theory(similar to the proof of \cite[Theorem 4.1]{HL2}), we have
\begin{equation}\label{eq4.8*}
\overline{u}_{i,k}(x)\rightarrow0\ \mbox{as}\ |x|\rightarrow\infty\ \mbox{uniformly on}\ k.
\end{equation}
Using \eqref{eq4.4} and \eqref{eq4.8}--\eqref{eq4.8*}, we obtain that there exists a constant $R>0$ large enough such that
\begin{equation*}
-\Delta \overline{u}_{i,k}+\frac{1}{2}\overline{u}_{i,k}\leq0\ \mbox{and}\ \overline{u}_{i,k}\leq Ce^{-\frac{1}{2}R}\ \mbox{for}\ |x|\geq R,
\end{equation*}
where $C>0$ is a constant independent of $k$. Comparing $\overline{u}_{i,k}$ with $e^{-\frac{1}{2}|x|}$, and the comparison principle implies that
\begin{equation*}
\overline{u}_{i,k}(x)\leq Ce^{-\frac{|x|}{2}}\ \mbox{for}\ |x|\geq R.
\end{equation*}
Furthermore, since $V(x)$ satisfies \eqref{eq1.22}, we have
\begin{equation*}
|\overline{\epsilon}_{k}^{2}V(\overline{\epsilon}_{k}x+z_{1,k})\overline{u}_{i,k}|\leq C^{}e^{-\frac{|x|}{4}}\ \mbox{for}\ |x|\geq R.
\end{equation*}
Then, applying the local elliptic estimates (e.g., $(3.15)$ in \cite{GT}) and the above estimates yield that
\begin{equation*}
|\nabla \overline{u}_{i,k}(x)|\leq Ce^{-\frac{|x|}{4}}\ \mbox{for}\ |x|\geq R.
\end{equation*}
\end{proof}
The next lemma is to give the limit behavior of $\overline{\eta}_{k}$.
\begin{lemma}\label{lemma4.2}
If $a\geq a^\ast$ and all the assumptions of Theorem \ref{th4} hold, then, passing to a subsequence, there exists $\overline{\eta}_{0} \in C^1_{loc}(\mathbb{R}^2)$ such that
\begin{equation}\label{eq4.9}
\overline{\eta}_{k}\rightarrow\overline{\eta}_{0}\ \mbox{in}\ C_{loc}^{1}(\mathbb R^2)\ \mbox{as}\ k\rightarrow\infty.
\end{equation}
Moreover,
\begin{equation}\label{eq4.10}
\overline{\eta}_{0}(x)=d_{0}(Q+x\cdot\nabla Q)+\sum_{i=1}^{2}d_{i}\frac{\partial Q}{\partial x_{i}},\ \mbox{if}\ a=a^\ast,
\end{equation}
and
\begin{equation}\label{eq4.10*}
\overline{\eta}_{0}(x)=h_{0}Q+\overline{h}_{0}x\cdot\nabla Q+\sum_{i=1}^{2}h_{i}\frac{\partial Q}{\partial x_{i}},\ \mbox{if}\ a>a^\ast,
\end{equation}
where $d_{0},d_{1},d_{2}$ and $h_{0},\overline{h}_{0},h_{1},h_{2}$ are some constants.
\end{lemma}
\begin{proof}[\textup{\textbf{Proof}}]
Since $||\overline{\eta}_{k}||_{L^\infty}=1$, it follows from \eqref{eq4.3}, \eqref{eq4.5}--\eqref{eq4.6} and the standard elliptic regularity theory that there exists $C>0$, independent of $k$, such that
 \begin{equation*}
 ||\overline{\eta}_{k}||_{C_{loc}^{1,\alpha}(\mathbb R^2)}\leq C\ \mbox{for some}\ \alpha\in(0,1).
 \end{equation*}
 Therefore, passing to a subsequence, there exists some function $\overline{\eta}_{0}(x)\in C_{loc}^{1}(\mathbb R^2)$ such that
 \begin{equation*}
 \overline{\eta}_{k}\rightarrow\overline{\eta}_{0}\ \mbox{in}\ C_{loc}^{1}(\mathbb R^2)\ \mbox{as}\ k\rightarrow\infty.
 \end{equation*}
Applying \eqref{eq4.1} and \eqref{eq4.2}, we know that
\begin{equation}\label{eq4.11}
\mu_{i,k}\overline{\epsilon}_{k}^{2}=2\overline{\epsilon}_{k}^{2}e_{a}(b_k)
+\frac{b_{k}}{2\overline{\epsilon}_{k}^{2}(a^{\ast})^{2}}\left(\int_{\mathbb R^{2}}|\nabla \overline{u}_{i,k}|^{2}dx\right)^{2}
-\frac{a}{2(a^\ast)^{2}}\int_{\mathbb R^{2}}\overline{u}_{i,k}^{4}dx,\ i=1,2.
\end{equation}
This implies that
\begin{equation}\label{eq4.12}
\aligned
\overline{g}_{k}(x)&=\overline{\epsilon}_{k}^{2}\frac{\mu_{1,k}-\mu_{2,k}}{||\overline{u}_{1,k}-\overline{u}_{2,k}||_{L^{\infty}}}\overline{u}_{2,k}\\
&=\frac{b_{k}\overline{u}_{2,k}}{2\overline{\epsilon}_{k}^{2}(a^{\ast})^{2}}\int_{\mathbb R^{2}}(|\nabla \overline{u}_{1,k}|^{2}+|\nabla \overline{u}_{2,k}|^{2})dx\int_{\mathbb R^{2}}\nabla(\overline{u}_{1,k}+\overline{u}_{2,k})\cdot\nabla \overline{\eta}_{k}dx\\
&\quad-\frac{a\overline{u}_{2,k}}{2(a^\ast)^{2}}\int_{\mathbb R^{2}}(\overline{u}_{1,k}^2+\overline{u}_{2,k}^2)(\overline{u}_{1,k}+\overline{u}_{2,k})\overline{\eta}_{k}dx.
\endaligned
\end{equation}
By the definition of \eqref{eq1.20}, we see that
\begin{equation}\label{eq4.13}
\frac{b_{k}}{\overline{\epsilon}_{k}^{2}}=\frac{a-a^\ast}{a^\ast},\ \mbox{if}\ a>a^\ast,\ \mbox{and}\ \  \frac{b_{k}}{\overline{\epsilon}_{k}^{2}}\rightarrow0\ \mbox{as}\ k\rightarrow\infty\ \mbox{if}\ a=a^\ast.
\end{equation}
Thus, let $k\to\infty$ in \eqref{eq4.5}, it follows from \eqref{eq4.6}, \eqref{eq4.12} and \eqref{eq4.13} that $\overline{\eta}_{0}$ satisfies
\begin{equation}\label{eq4.14}
-\Delta\overline{\eta}_{0}+(1-3Q^{2})\overline{\eta}_{0}=-\frac{2Q}{a^\ast}\int_{\mathbb R^{2}}Q^{3}\overline{\eta}_{0}dx,\ \mbox{if}\ a=a^{\ast},
\end{equation}
and
\begin{equation}\label{eq4.14*}
\aligned
-\Delta\overline{\eta}_{0}+(1-3Q^{2})\overline{\eta}_{0}&=\frac{2(a-a^\ast)}{aa^\ast}\int_{\mathbb R^2}\nabla Q\cdot\nabla\overline{\eta}_{0}dx\Delta Q+
\frac{2(a-a^\ast)Q}{aa^\ast}\int_{\mathbb R^2}\nabla Q\cdot\nabla\overline{\eta}_{0}dx\\
&\quad-\frac{2Q}{a^\ast}\int_{\mathbb R^{2}}Q^{3}\overline{\eta}_{0}dx,\ \mbox{if}\ a>a^{\ast}.
\endaligned
\end{equation}
Let $\Gamma:=-\Delta+(1-3Q^{2})$, and it is easy check that $\Gamma(Q+x\cdot\nabla Q)=-2Q$ and $\Gamma(x\cdot\nabla Q)=-2\Delta Q$. Moreover, recall from \cite{KMK,NT} that
\begin{equation*}
\mbox{ker}\ \Gamma=\mbox{span}\left\{\frac{\partial Q}{\partial x_{1}},\frac{\partial Q}{\partial x_{2}}\right\}.
\end{equation*}
Then, \eqref{eq4.14} and \eqref{eq4.14*} imply that
\begin{equation*}
  \overline{\eta}_{0}(x)=d_{0}(Q+x\cdot\nabla Q)+\sum_{i=1}^{2}d_{i}\frac{\partial Q}{\partial x_{i}},\ \mbox{if}\ a=a^\ast,
\end{equation*}
and
\begin{equation*}
\overline{\eta}_{0}(x)=h_{0}Q+\overline{h}_{0}x\cdot\nabla Q+\sum_{i=1}^{2}h_{i}\frac{\partial Q}{\partial x_{i}},\ \mbox{if}\ a>a^\ast,
\end{equation*}
where $d_{0},d_{1}$,$d_{2}$ and $h_{0},\overline{h}_{0},h_{1},h_{2}$ are constants.
\end{proof}
\begin{lemma}\label{lemma4.3}
If $a\geq a^\ast$ and all assumptions of Theorem \ref{th4} hold, then

\begin{equation}\label{eq4.15}
d_{0}\int_{\mathbb R^{2}}\frac{\partial V_{1}(x+y_{0})}{\partial x_{j}}(x\cdot\nabla Q^{2})dx
-\sum_{i=1}^{2}d_{i}\int_{\mathbb R^{2}}\frac{\partial^{2} V_{1}(x+y_{0})}{\partial x_{i}\partial x_{j}}Q^{2}dx=0,\ \mbox{if}\ a=a^\ast,
\end{equation}
and
\begin{equation}\label{eq4.15*}
\overline{h}_{0}\int_{\mathbb R^{2}}\frac{\partial V_{1}(x+y_{0})}{\partial x_{j}}(x\cdot\nabla Q^{2})dx
-\sum_{i=1}^{2}h_{i}\int_{\mathbb R^{2}}\frac{\partial^{2} V_{1}(x+y_{0})}{\partial x_{i}\partial x_{j}}Q^{2}dx=0,\ \mbox{if}\ a>a^\ast,
\end{equation}
where $V_{1}(x)$ is given by \eqref{eq1.14} and $j=1,2$.
\end{lemma}
\begin{proof}[\textup{\textbf{Proof}}]
Denote
\begin{equation}\label{eq4.16}
\widehat{u}_{i,k}(x)=\sqrt{a^\ast}\overline{\epsilon}_{k}u_{i,k}(x)\ \mbox{and}\ \widehat{\eta}_{k}=\frac{\widehat{u}_{1,k}-\widehat{u}_{2,k}}{||\widehat{u}_{1,k}-\widehat{u}_{2,k}||_{L^{\infty}}},\ i=1,2.
\end{equation}
It follows from \eqref{eq4.2} that
\begin{equation}\label{eq4.17}
\overline{u}_{i,k}(x)=\widehat{u}_{i,k}(\overline{\epsilon}_{k}x+z_{1,k})\rightarrow Q(x)\ \mbox{uniformly in}\ x\in\mathbb R^2\ \mbox{as}\ k\rightarrow\infty.
\end{equation}

For each $k>0$ and $x_{k}\in\mathbb R^2$, we claim that, for any fixed small $\overline{\delta}>0$ independence of $k$ and $x_{k}$, there exists a small constant $\delta_{k}\in(\overline{\delta},2\overline{\delta})$ such that
\begin{equation}\label{eq4.18}
\overline{\epsilon}_{k}^{2}\int_{\partial B_{\delta_{k}}(x_{k})}|\nabla \widehat{\eta}_{k}|^{2}dS
+\overline{\epsilon}_{k}^{2}\int_{\partial B_{\delta_{k}}(x_{k})}V(x)\widehat{\eta}_{k}^{2}dS
+\int_{\partial B_{\delta_{k}}(x_{k})}\widehat{\eta}_{k}^{2}dS\leq O(\overline{\epsilon}_{k}^{2})\ \mbox{as}\ k\rightarrow\infty.
\end{equation}
Similar to \eqref{eq4.4}--\eqref{eq4.6}, it follows from \eqref{eq4.1} that $\widehat{u}_{i,k}$ and $\widehat{\eta}_{k}$ satisfy
\begin{equation}\label{eq4.19}
-\left(\overline{\epsilon}_{k}^{2}+\frac{b_{k}}{a^{\ast}}\int_{\mathbb R^{2}}|\nabla \widehat{u}_{i,k}|^{2}dx\right)\Delta \widehat{u}_{i,k}+\overline{\epsilon}_{k}^{2}V(x)\widehat{u}_{i,k}
=\mu_{i,k}\overline{\epsilon}_{k}^{2} \widehat{u}_{i,k}+\frac{a}{a^\ast}\widehat{u}_{i,k}^{3},\  x\in \mathbb{R}^{2},\ i=1,2
\end{equation}
and
\begin{align}\label{eq4.20}
-\left(2\overline{\epsilon}_{k}^{2}+\frac{b_{k}}{a^{\ast}}\int_{\mathbb R^{2}}(|\nabla \widehat{u}_{1,k}|^{2}+|\nabla \widehat{u}_{2,k}|^{2})dx\right)\Delta \widehat{\eta}_{k}
& +2\overline{\epsilon}_{k}^{2}V(x)\widehat{\eta}_{k} 
 -\frac{b_{k}}{a^{\ast}}\int_{\mathbb R^{2}}\nabla(\widehat{u}_{1,k}
+\widehat{u}_{2,k})\cdot\nabla\widehat{\eta}_{k}dx\Delta(\widehat{u}_{1,k}+\widehat{u}_{2,k})\nonumber\\
&=(\mu_{1,k}+\mu_{2,k})\overline{\epsilon}_{k}^{2} \widehat{\eta}_{k}+\widehat{g}_{k}(x)+\widehat{f}_{k}(x),
\end{align}
where
\begin{equation}\label{eq4.21}
\widehat{g}_{k}(x)=\overline{\epsilon}_{k}^{2}\frac{\mu_{1,k}-\mu_{2,k}}{||\widehat{u}_{1,k}-\widehat{u}_{2,k}||_{L^{\infty}}}(\widehat{u}_{1,k}
+\widehat{u}_{2,k})\ \mbox{and}\ \widehat{f}_{k}(x)=\frac{2a\widehat{\eta}_{k}}{a^\ast}(\widehat{u}_{1,k}^{2}+\widehat{u}_{1,k}\widehat{u}_{2,k}+\widehat{u}_{2,k}^{2}).
\end{equation}
Moreover,
\begin{equation}\label{eq4.22}
\aligned
\widehat{g}_{k}(x)&=\overline{\epsilon}_{k}^{2}\frac{\mu_{1,k}-\mu_{2,k}}{||\widehat{u}_{1,k}-\widehat{u}_{2,k}||_{L^{\infty}}}(\widehat{u}_{1,k}
+\widehat{u}_{2,k})\\
&=\frac{b_{k}}{2\overline{\epsilon}_{k}^{2}(a^{\ast})^{2}}\int_{\mathbb R^{2}}(|\nabla \overline{u}_{1,k}|^{2}+|\nabla \overline{u}_{2,k}|^{2})dx\int_{\mathbb R^{2}}\nabla(\overline{u}_{1,k}+\overline{u}_{2,k})\cdot\nabla \overline{\eta}_{k}dx(\widehat{u}_{1,k}
+\widehat{u}_{2,k})\\
&\quad -\frac{a}{2(a^\ast)^{2}}\int_{\mathbb
R^{2}}(\overline{u}_{1,k}^2+\overline{u}_{2,k}^2)(\overline{u}_{1,k}+\overline{u}_{2,k})\overline{\eta}_{k}dx(\widehat{u}_{1,k}
+\widehat{u}_{2,k}).
\endaligned
\end{equation}
Multiplying \eqref{eq4.20} by $\widehat{\eta}_{k}$ and integrating over $\mathbb{R}^{2}$ we have
\begin{equation}\label{eq4.23}
\aligned
&\left(2\overline{\epsilon}_{k}^{2}+\frac{b_{k}}{a^{\ast}}\int_{\mathbb R^{2}}(|\nabla \widehat{u}_{1,k}|^{2}+|\nabla \widehat{u}_{2,k}|^{2})dx\right)\int_{\mathbb R^{2}}|\nabla\widehat{\eta}_{k}|^{2}dx
+\frac{b_{k}}{a^{\ast}}\left(\int_{\mathbb R^{2}}\nabla(\widehat{u}_{1,k}
+\widehat{u}_{2,k})\cdot\nabla\widehat{\eta}_{k}dx\right)^{2}\\
&\quad+2\overline{\epsilon}_{k}^{2}\int_{\mathbb R^{2}}V(x)\widehat{\eta}_{k}^{2}dx-
(\mu_{1,k}+\mu_{2,k})\overline{\epsilon}_{k}^{2}\int_{\mathbb R^{2}}\widehat{\eta}_{k}^{2}dx\\
&=\frac{b_{k}}{2(a^{\ast})^{2}}\int_{\mathbb R^{2}}(|\nabla \overline{u}_{1,k}|^{2}+|\nabla \overline{u}_{2,k}|^{2})dx\int_{\mathbb R^{2}}\nabla(\overline{u}_{1,k}+\overline{u}_{2,k})\cdot\nabla \overline{\eta}_{k}dx\int_{\mathbb R^{2}}(\overline{u}_{1,k}+\overline{u}_{2,k})\overline{\eta}_{k}dx\\
&\quad-\frac{a\overline{\epsilon}_{k}^{2}}{2(a^\ast)^{2}}\int_{\mathbb
R^{2}}(\overline{u}_{1,k}^2+\overline{u}_{2,k}^2)(\overline{u}_{1,k}+\overline{u}_{2,k})\overline{\eta}_{k}dx\int_{\mathbb R^{2}}(\overline{u}_{1,k}+\overline{u}_{2,k})\overline{\eta}_{k}dx\\
&\quad+\frac{2a\overline{\epsilon}_{k}^{2}}{a^\ast}\int_{\mathbb R^{2}}(\overline{u}_{1,k}^{2}+\overline{u}_{1,k}\overline{u}_{2,k}+\overline{u}_{2,k}^{2})\overline{\eta}_{k}^{2}dx
=O(\overline{\epsilon}_{k}^{2})
\endaligned
\end{equation}
since $|\widehat{\eta}_{k}|$ is bounded uniformly in $k$, and $\overline{u}_{i,k}$ decay exponentially as $|x|\rightarrow\infty$, $i=1,2$. Then, \eqref{eq1.20} and \eqref{eq4.8} mean that
\begin{equation}\label{eq4.24}
\overline{\epsilon}_{k}^{2}\int_{\mathbb R^{2}}|\nabla \widehat{\eta}_{k}|^{2}dx
+\overline{\epsilon}_{k}^{2}\int_{\mathbb R^{2}}V(x)\widehat{\eta}_{k}^{2}dx
+\int_{\mathbb R^{2}}\widehat{\eta}_{k}^{2}dx\leq O(\overline{\epsilon}_{k}^{2})\ \mbox{as}\ k\rightarrow\infty.
\end{equation}
This and Lemma A.4 in \cite{LPW} show that for any fixed small $\overline{\delta}>0$ independence of $k$ and $x_{k}$, there exists a small constant $\delta_{k}\in(\overline{\delta},2\overline{\delta})$ such that
 \begin{equation*}
  \overline{\epsilon}_{k}^{2}\int_{\partial B_{\delta_{k}}(x_{k})}|\nabla \widehat{\eta}_{k}|^{2}dS
+\overline{\epsilon}_{k}^{2}\int_{\partial B_{\delta_{k}}(x_{k})}V(x)\widehat{\eta}_{k}^{2}dS
+\int_{\partial B_{\delta_{k}}(x_{k})}\widehat{\eta}_{k}^{2}dS\leq O(\overline{\epsilon}_{k}^{2})\ \mbox{as}\ k\rightarrow\infty.
 \end{equation*}
Hence, \eqref{eq4.18} is proved.

Multiplying \eqref{eq4.19} by $\frac{\partial\widehat{u}_{i,k}}{\partial x_{j}}$ and integrating over $B_{\delta_{k}}(z_{1,k})$, where $i,j=1,2,j=1,2$, and $\delta_{k}$ is given by \eqref{eq4.18}, we see that
\begin{equation}\label{eq4.25}
\aligned
&-\left(\overline{\epsilon}_{k}^{2}+\frac{b_{k}}{a^{\ast}}\int_{\mathbb R^{2}}|\nabla \widehat{u}_{i,k}|^{2}dx\right)\int_{B_{\delta_{k}}(z_{1,k})}\Delta \widehat{u}_{i,k}\frac{\partial\widehat{u}_{i,k}}{\partial x_{j}}dx
+\frac{\overline{\epsilon}_{k}^{2}}{2}\int_{B_{\delta_{k}}(z_{1,k})}V(x)\frac{\partial\widehat{u}_{i,k}^{2}}{\partial x_{j}}dx\\
&\quad=\frac{\mu_{i,k}\overline{\epsilon}_{k}^{2}}{2} \int_{B_{\delta_{k}}(z_{1,k})}\frac{\partial\widehat{u}_{i,k}^{2}}{\partial x_{j}}dx+\frac{a}{4a^\ast}\int_{B_{\delta_{k}}(z_{1,k})}\frac{\partial\widehat{u}_{i,k}^{4}}{\partial x_{j}}dx.
\endaligned
\end{equation}
By calculations, we know that
\begin{equation}\label{eq4.26}
\aligned
&\int_{B_{\delta_{k}}(z_{1,k})}\Delta \widehat{u}_{i,k}\frac{\partial\widehat{u}_{i,k}}{\partial x_{j}}dx
=\sum_{m=1}^{2}\int_{B_{\delta_{k}}(z_{1,k})}\frac{\partial^{2}\widehat{u}_{i,k}}{\partial x_{m}^{2}}\frac{\partial\widehat{u}_{i,k}}{\partial x_{j}}dx\\
&=\sum_{m=1}^{2}\int_{B_{\delta_{k}}(z_{1,k})}\frac{\partial}{\partial x_{m}}\left(\frac{\partial\widehat{u}_{i,k}}{\partial x_{m}}\frac{\partial\widehat{u}_{i,k}}{\partial x_{j}}\right)dx
-\sum_{m=1}^{2}\int_{B_{\delta_{k}}(z_{1,k})}\frac{\partial\widehat{u}_{i,k}}{\partial x_{m}}\frac{\partial^{2}\widehat{u}_{i,k}}{\partial x_{j}\partial x_{m}}dx\\
&=\sum_{m=1}^{2}\int_{B_{\delta_{k}}(z_{1,k})}\frac{\partial}{\partial x_{m}}\left(\frac{\partial\widehat{u}_{i,k}}{\partial x_{m}}\frac{\partial\widehat{u}_{i,k}}{\partial x_{j}}\right)dx
-\frac{1}{2}\sum_{m=1}^{2}\int_{B_{\delta_{k}}(z_{1,k})}\frac{\partial}{\partial x_{j}}\left(\frac{\partial\widehat{u}_{i,k}}{\partial x_{m}}\right)^{2}dx\\
&=\int_{\partial B_{\delta_{k}}(z_{1,k})}\frac{\partial\widehat{u}_{i,k}}{\partial x_{j}}\frac{\partial\widehat{u}_{i,k}}{\partial \nu}dS
-\frac{1}{2}\int_{\partial B_{\delta_{k}}(z_{1,k})}|\nabla \widehat{u}_{i,k}|^{2}\nu_{j}dS,
\endaligned
\end{equation}
and
\begin{equation}\label{eq4.27}
\int_{B_{\delta_{k}}(z_{1,k})}V(x)\frac{\partial\widehat{u}_{i,k}^{2}}{\partial x_{j}}dx
=\int_{\partial B_{\delta_{k}}(z_{1,k})}V(x)\widehat{u}_{i,k}^{2}\nu_{j}dS
-\int_{B_{\delta_{k}}(z_{1,k})}\frac{\partial V(x)}{\partial x_{j}}\widehat{u}_{i,k}^{2}dx.
\end{equation}
Then, by \eqref{eq4.25}--\eqref{eq4.27} we have
\begin{equation*}
\aligned
 \frac{\overline{\epsilon}_{k}^{2}}{2}\int_{B_{\delta_{k}}(z_{1,k})}\frac{\partial V(x)}{\partial x_{j}}\widehat{u}_{i,k}^{2}dx
 = & \ \frac{\overline{\epsilon}_{k}^{2}}{2}\int_{\partial B_{\delta_{k}}(z_{1,k})}V(x)\widehat{u}_{i,k}^{2}\nu_{j}dS
 -\frac{\mu_{i,k}\overline{\epsilon}_{k}^{2}}{2} \int_{\partial B_{\delta_{k}}(z_{1,k})}\widehat{u}_{i,k}^{2}\nu_{j}dS
 -\frac{a}{4a^\ast}\int_{\partial B_{\delta_{k}}(z_{1,k})}\widehat{u}_{i,k}^{4}\nu_{j}dS\\
&\quad -\left(\overline{\epsilon}_{k}^{2}+\frac{b_{k}}{a^{\ast}}\int_{\mathbb R^{2}}|\nabla \widehat{u}_{i,k}|^{2}dx\right)
 \left[\int_{\partial B_{\delta_{k}}(z_{1,k})}\frac{\partial\widehat{u}_{i,k}}{\partial x_{j}}\frac{\partial\widehat{u}_{i,k}}{\partial \nu}dS
-\frac{1}{2}\int_{\partial B_{\delta_{k}}(z_{1,k})}|\nabla \widehat{u}_{i,k}|^{2}\nu_{j}dS\right].
 \endaligned
\end{equation*}
Moreover, we have
\begin{equation}\label{eq4.28}
\frac{\overline{\epsilon}_{k}^{2}}{2}\int_{B_{\delta_{k}}(z_{1,k})}\frac{\partial V(x)}{\partial x_{j}}(\widehat{u}_{1,k}+\widehat{u}_{2,k})\widehat{\eta}_{k}dx
=I_{1}+I_{2}+I_{3}+I_{4}+I_{5}+I_{6},
\end{equation}
where
\begin{equation}\label{eq4.29}
I_{1}=\frac{\overline{\epsilon}_{k}^{2}}{2}\int_{\partial B_{\delta_{k}}(z_{1,k})}V(x)(\widehat{u}_{1,k}+\widehat{u}_{2,k})\widehat{\eta}_{k}\nu_{j}dS,
\end{equation}
\begin{equation}\label{eq4.30}
I_{2}=
-\frac{\mu_{1,k}\overline{\epsilon}_{k}^{2}}{2} \int_{\partial B_{\delta_{k}}(z_{1,k})}(\widehat{u}_{1,k}+\widehat{u}_{2,k})\widehat{\eta}_{k}\nu_{j}dS
-\frac{\overline{\epsilon}_{k}^{2}(\mu_{1,k}-\mu_{2,k})}{2||\widehat{u}_{1,k}-\widehat{u}_{2,k}||_{L^{\infty}}}\int_{\partial B_{\delta_{k}}(z_{1,k})}\widehat{u}_{2,k}^{2}\nu_{j}dS,
\end{equation}
\begin{equation}\label{4.31}
I_{3}=-\frac{a}{4a^\ast}\int_{\partial B_{\delta_{k}}(z_{1,k})}(\widehat{u}_{1,k}+\widehat{u}_{2,k})(\widehat{u}_{1,k}^{2}+\widehat{u}_{2,k}^{2})\widehat{\eta}_{k}\nu_{j}dS,
\end{equation}
\begin{equation}\label{eq4.32}
I_{4}=-\left(\overline{\epsilon}_{k}^{2}+\frac{b_{k}}{a^{\ast}}\int_{\mathbb R^{2}}|\nabla \widehat{u}_{1,k}|^{2}dx\right)
\left[\int_{\partial B_{\delta_{k}}(z_{1,k})}\frac{\partial\widehat{\eta}_{k}}{\partial x_{j}}\frac{\partial\widehat{u}_{1,k}}{\partial \nu}dS+
\int_{\partial B_{\delta_{k}}(z_{1,k})}\frac{\partial\widehat{u}_{2,k}}{\partial x_{j}}\frac{\partial\widehat{\eta}_{k}}{\partial \nu}dS
\right],
\end{equation}
\begin{equation}\label{eq4.33}
I_{5}=\left(\frac{\overline{\epsilon}_{k}^{2}}{2}+\frac{b_{k}}{2a^{\ast}}\int_{\mathbb R^{2}}|\nabla \widehat{u}_{1,k}|^{2}dx\right)
\int_{\partial B_{\delta_{k}}(z_{1,k})}\nabla(\widehat{u}_{1,k}+\widehat{u}_{2,k})\cdot\nabla \widehat{\eta}_{k}\nu_{j}dS,
\end{equation}
and
\begin{equation}\label{eq4.34}
I_{6}=-\frac{b_{k}}{a^{\ast}}\int_{\mathbb R^{2}}\nabla(\widehat{u}_{1,k}+\widehat{u}_{2,k})\cdot\nabla \widehat{\eta}_{k}dx
\left[\int_{\partial B_{\delta_{k}}(z_{1,k})}\frac{\partial\widehat{u}_{2,k}}{\partial x_{j}}\frac{\partial\widehat{u}_{2,k}}{\partial \nu}dS
-\frac{1}{2}\int_{\partial B_{\delta_{k}}(z_{1,k})}|\nabla \widehat{u}_{2,k}|^{2}\nu_{j}dS\right].
\end{equation}
Using \eqref{eq4.13} and \eqref{eq4.22}, there exists a constant $C>0$, independent of $k$, such that
\begin{equation}\label{eq4.34*}
\left|\overline{\epsilon}_{k}^{2}\frac{\mu_{1,k}-\mu_{2,k}}{||\widehat{u}_{1,k}-\widehat{u}_{2,k}||_{L^{\infty}}}\right|\leq C.
\end{equation}
Then, using the H\"{o}lder's inequality, we can derive from \eqref{eq4.7}, \eqref{eq4.18} and the above estimates that there exists a constant $C>0$, independent of $k$, such that
\begin{equation}\label{eq4.35}
|I_{1}|\leq \frac{\overline{\epsilon}_{k}^{2}}{2}\left(\int_{\partial B_{\delta_{k}}(z_{1,k})}V(x)\widehat{\eta}_{k}^{2}dS\right)^{\frac{1}{2}}
\left(\int_{\partial B_{\delta_{k}}(z_{1,k})}V(x)(\widehat{u}_{1,k}+\widehat{u}_{2,k})^{2}dS\right)^{\frac{1}{2}}=o( e^{-\frac{C\delta_{k}}{\overline{\epsilon}_{k}}}).
\end{equation}
Similarly, using \eqref{eq4.7}, \eqref{eq4.8} and \eqref{eq4.18} again, one can obtain that
\begin{equation}\label{eq4.35*}
|I_{2}|\leq C\left(\int_{\partial B_{\delta_{k}}(z_{1,k})}(\widehat{u}_{1,k}+\widehat{u}_{2,k})^{2}dS\right)^{\frac{1}{2}}
\left(\int_{\partial B_{\delta_{k}}(z_{1,k})}|\widehat{\eta}_{k}|^{2}dx\right)^{\frac{1}{2}}
+C\int_{\partial B_{\delta_{k}}(z_{1,k})}|\widehat{u}_{2,k}|^{2}dS
=o(e^{-\frac{C\delta_{k}}{\epsilon_k}}),
\end{equation}
\begin{equation}\label{eq4.35**}
|I_{3}|\leq C\int_{\partial B_{\delta_{k}}(z_{1,k})}(\widehat{u}_{1,k}+\widehat{u}_{2,k})(\widehat{u}_{1,k}^{2}+\widehat{u}_{2,k}^{2})dS=o(e^{-\frac{C\delta_{k}}{\epsilon_k}}),
\end{equation}
\begin{equation}\label{eq4.35***}
\aligned
|I_{4}|\leq o(1)\left(\int_{\partial B_{\delta_{k}}(z_{1,k})}|\nabla \widehat{\eta}_{k}|^{2}dS\right)^{\frac{1}{2}}
\left[\left(\int_{\partial B_{\delta_{k}}(z_{1,k})}|\nabla \widehat{u}_{1,k}|^{2}dS\right)^{\frac{1}{2}}
+\left(\int_{\partial B_{\delta_{k}}(z_{1,k})}|\nabla \widehat{u}_{2,k}|^{2}dS\right)^{\frac{1}{2}}\right]
=o(e^{-\frac{C\delta_{k}}{\epsilon_k}}),
\endaligned
\end{equation}
\begin{equation}\label{eq4.35****}
|I_{5}|\leq C\left(\int_{\partial B_{\delta_{k}}(z_{1,k})}|\nabla(\widehat{u}_{1,k}+\widehat{u}_{2,k})|^{2}dS\right)^{\frac{1}{2}}
\left(\int_{\partial B_{\delta_{k}}(z_{1,k})}|\nabla\widehat{\eta}_{k}|^{2}dx\right)^{\frac{1}{2}}=o(e^{-\frac{C\delta_{k}}{\epsilon_k}}),
\end{equation}
and
\begin{equation}\label{eq4.36}
|I_{6}|\leq o(1)\int_{\partial B_{\delta_{k}}(z_{1,k})}|\nabla \widehat{u}_{2,k}|^{2}dS=o(e^{-\frac{C\delta_{k}}{\epsilon_k}}).
\end{equation}
Applying \eqref{eq1.21}, \eqref{eq1.23} and \eqref{eq4.35}--\eqref{eq4.36}, we can deduce from \eqref{eq4.28} that
\begin{equation}\label{eq4.37}
\aligned
o( e^{-\frac{C\delta_{k}}{\overline{\epsilon}_{k}}})&=\frac{\overline{\epsilon}_{k}^{2}}{2}\int_{B_{\delta_{k}}(z_{1,k})}\frac{\partial V(x)}{\partial x_{j}}(\widehat{u}_{1,k}+\widehat{u}_{2,k})\widehat{\eta}_{k}dx\\
&=\frac{\overline{\epsilon}_{k}^{4}}{2}\int_{B_{\frac{\delta_{k}}{\overline{\epsilon}_{k}}}(0)}\frac{\partial V(\overline{\epsilon}_{k}[x+(z_{1,k}-x_{1})/\overline{\epsilon}_{k}]+x_{1})}{\overline{\epsilon}_{k}\partial x_{j}}(\overline{u}_{1,k}+\overline{u}_{2,k})\overline{\eta}_{k}dx\\
&=\frac{\overline{\epsilon}_{k}^{p+3}}{2}\int_{B_{\frac{\delta_{k}}{\overline{\epsilon}_{k}}}(0)}
\frac{\partial V_{1}(x+(z_{1,k}-x_{1})/\overline{\epsilon}_{k})}{\partial x_{j}}(\overline{u}_{1,k}+\overline{u}_{2,k})\overline{\eta}_{k}dx\\
&\quad+\frac{\overline{\epsilon}_{k}^{4}}{2}\int_{B_{\frac{\delta_{k}}{\overline{\epsilon}_{k}}}(0)}
W_{j}(\overline{\epsilon}_{k}x+z_{1,k}-x_{1})(\overline{u}_{1,k}+\overline{u}_{2,k})\overline{\eta}_{k}dx\\
&=(1+o(1))\overline{\epsilon}_{k}^{p+3}\int_{\mathbb R^2}\frac{\partial V_{1}(x+y_{0})}{\partial x_{j}}Q\overline{\eta}_{0}dx.
\endaligned
\end{equation}
Therefore, \eqref{eq4.15} and \eqref{eq4.15*} can be obtained by considering the cases $a=a^{\ast}$ and $a>a^{\ast}$, respectively.

If $a=a^\ast$, it follows from \eqref{eq1.21}, \eqref{eq4.10} and \eqref{eq4.37} that
\begin{align*}
0&=\int_{\mathbb R^2}\frac{\partial V_{1}(x+y_{0})}{\partial x_{j}}Q\overline{\eta}_{0}dx
=\int_{\mathbb R^2}\frac{\partial V_{1}(x+y_{0})}{\partial x_{j}}Q\left[d_{0}(Q+x\cdot\nabla Q)+\sum_{i=1}^{2}d_{i}\frac{\partial Q}{\partial x_{i}}\right]dx\\
&=d_{0}\int_{\mathbb R^2}\frac{\partial V_{1}(x+y_{0})}{\partial x_{j}}Q^{2}dx
+\frac{d_{0}}{2}\int_{\mathbb R^2}\frac{\partial V_{1}(x+y_{0})}{\partial x_{j}}(x\cdot\nabla Q^{2})dx 
+\frac{1}{2}\sum_{i=1}^{2}d_{i}\int_{\mathbb R^2}\frac{\partial V_{1}(x+y_{0})}{\partial x_{j}}\frac{\partial Q^{2}}{\partial x_{i}}dx\\
&=\frac{d_{0}}{2}\int_{\mathbb R^2}\frac{\partial V_{1}(x+y_{0})}{\partial x_{j}}(x\cdot\nabla Q^{2})dx
-\frac{1}{2}\sum_{i=1}^{2}d_{i}\int_{\mathbb R^2}\frac{\partial^{2} V_{1}(x+y_{0})}{\partial x_{i}\partial x_{j}}Q^{2}dx,
\end{align*}
which gives \eqref{eq4.15}.

If $a>a^\ast$, it follows from \eqref{eq1.21}, \eqref{eq4.10*} and \eqref{eq4.37} that
\begin{align*}
0&=\int_{\mathbb R^2}\frac{\partial V_{1}(x+y_{0})}{\partial x_{j}}Q\overline{\eta}_{0}dx
=\int_{\mathbb R^2}\frac{\partial V_{1}(x+y_{0})}{\partial x_{j}}Q\left(h_{0}Q+\overline{h}_{0}x\cdot\nabla Q+\sum_{i=1}^{2}h_{i}\frac{\partial Q}{\partial x_{i}}\right)dx\\
&=h_{0}\int_{\mathbb R^2}\frac{\partial V_{1}(x+y_{0})}{\partial x_{j}}Q^{2}dx
+\frac{\overline{h}_{0}}{2}\int_{\mathbb R^2}\frac{\partial V_{1}(x+y_{0})}{\partial x_{j}}(x\cdot\nabla Q^{2})dx 
+\frac{1}{2}\sum_{i=1}^{2}h_{i}\int_{\mathbb R^2}\frac{\partial V_{1}(x+y_{0})}{\partial x_{j}}\frac{\partial Q^{2}}{\partial x_{i}}dx\\
&=\frac{\overline{h}_{0}}{2}\int_{\mathbb R^2}\frac{\partial V_{1}(x+y_{0})}{\partial x_{j}}(x\cdot\nabla Q^{2})dx
-\frac{1}{2}\sum_{i=1}^{2}h_{i}\int_{\mathbb R^2}\frac{\partial^{2} V_{1}(x+y_{0})}{\partial x_{i}\partial x_{j}}Q^{2}dx,
\end{align*}
which gives \eqref{eq4.15*}.

Finally, we give the proof of Theorem \ref{th4}.
\end{proof}
\begin{proof}[\textup{\textbf{Proof of Theorem \ref{th4}}}]
The key step of proving the theorem is to show that $d_{0}=0$ in \eqref{eq4.15} and $h_{0}=\overline{h}_{0}=0$ in \eqref{eq4.15*}.

For this purpose, multiplying \eqref{eq4.19} by $(x-z_{1,k})\cdot\nabla \widehat{u}_{i,k}$ and integrating over $B_{\delta_{k}}(z_{1,k})$ for $i=1,2$ and $\delta_{k}>0$ being small enough as before, then we have
\begin{align}\label{eq4.38}
&-\left(\overline{\epsilon}_{k}^{2}+\frac{b_{k}}{a^{\ast}}\int_{\mathbb R^{2}}|\nabla \widehat{u}_{i,k}|^{2}dx\right)
\int_{B_{\delta_{k}}(z_{1,k})}\Delta \widehat{u}_{i,k}[(x-z_{1,k})\cdot\nabla \widehat{u}_{i,k}]dx 
+\frac{\overline{\epsilon}_{k}^{2}}{2}\int_{B_{\delta_{k}}(z_{1,k})}V(x)[(x-z_{1,k})\cdot\nabla \widehat{u}_{i,k}^{2}]dx \nonumber \\
&=\frac{\mu_{i,k}\overline{\epsilon}_{k}^{2}}{2}\int_{B_{\delta_{k}}(z_{1,k})}(x-z_{1,k})\cdot\nabla \widehat{u}_{i,k}^{2}dx
+\frac{a}{4a^\ast}\int_{B_{\delta_{k}}(z_{1,k})}(x-z_{1,k})\cdot\nabla \widehat{u}_{i,k}^{4}dx.
\end{align}
Applying the integration by parts, we know that
\begin{equation}\label{eq4.39}
\aligned
T_{i}&\triangleq\int_{B_{\delta_{k}}(z_{1,k})}\Delta \widehat{u}_{i,k}[(x-z_{1,k})\cdot\nabla \widehat{u}_{i,k}]dx\\
&=\int_{\partial B_{\delta_{k}}(z_{1,k})}\frac{\partial\widehat{u}_{i,k}}{\partial\nu}[(x-z_{1,k})\cdot\nabla \widehat{u}_{i,k}]dS
-\int_{B_{\delta_{k}}(z_{1,k})}\nabla\widehat{u}_{i,k}\cdot\nabla[(x-z_{1,k})\cdot\nabla \widehat{u}_{i,k}]dx\\
&=\int_{\partial B_{\delta_{k}}(z_{1,k})}\frac{\partial\widehat{u}_{i,k}}{\partial\nu}[(x-z_{1,k})\cdot\nabla \widehat{u}_{i,k}]dS
-\frac{1}{2}\int_{\partial B_{\delta_{k}}(z_{1,k})}[(x-z_{1,k})\cdot\nu]|\nabla\widehat{u}_{i,k}|^{2}dS,
\endaligned
\end{equation}
\begin{equation}\label{eq4.40}
\aligned
\int_{B_{\delta_{k}}(z_{1,k})}V(x)[(x-z_{1,k})\cdot\nabla \widehat{u}_{i,k}^{2}]dx
& =\int_{\partial B_{\delta_{k}}(z_{1,k})}V(x)[(x-z_{1,k})\cdot\nu] \widehat{u}_{i,k}^{2}dS\\
& \hspace{1.5cm} -\int_{B_{\delta_{k}}(z_{1,k})}[\nabla V(x)\cdot(x-z_{1,k})+2V(x)]\widehat{u}_{i,k}^{2}dx,
\endaligned
\end{equation}
\begin{equation}\label{eq4.41}
\int_{B_{\delta_{k}}(z_{1,k})}(x-z_{1,k})\cdot\nabla \widehat{u}_{i,k}^{2}dx
=\int_{\partial B_{\delta_{k}}(z_{1,k})}[(x-z_{1,k})\cdot\nu] \widehat{u}_{i,k}^{2}dS
-2\int_{B_{\delta_{k}}(z_{1,k})}\widehat{u}_{i,k}^{2}dx,
\end{equation}
and
\begin{equation}\label{eq4.42}
\int_{B_{\delta_{k}}(z_{1,k})}(x-z_{1,k})\cdot\nabla \widehat{u}_{i,k}^{4}dx=
\int_{\partial B_{\delta_{k}}(z_{1,k})}[(x-z_{1,k})\cdot\nu] \widehat{u}_{i,k}^{4}dS
-2\int_{B_{\delta_{k}}(z_{1,k})}\widehat{u}_{i,k}^{4}dx.
\end{equation}
On the other hand, we know from \eqref{eq4.11} that
\begin{equation}\label{eq4.43}
\mu_{i,k}\overline{\epsilon}_{k}^{2}\int_{\mathbb R^2}\widehat{u}_{i,k}^{2}dx=2a^\ast\overline{\epsilon}_{k}^{4}e_{a}(b_k)
+\frac{b_k}{2a^\ast}\left(\int_{\mathbb R^{2}}|\nabla \widehat{u}_{i,k}|^{2}dx\right)^{2}
-\frac{a}{2a^\ast}\int_{\mathbb R^{2}}\widehat{u}_{i,k}^{4}dx.
\end{equation}
Then, \eqref{eq4.38}--\eqref{eq4.43} yield that
\begin{align}\label{eq4.44}
&-\left(\overline{\epsilon}_{k}^{2}+\frac{b_{k}}{a^{\ast}}\int_{\mathbb R^{2}}|\nabla \widehat{u}_{i,k}|^{2}dx\right)T_{i}
+\frac{\overline{\epsilon}_{k}^{2}}{2}\int_{\partial B_{\delta_{k}}(z_{1,k})}V(x)[(x-z_{1,k})\cdot\nu] \widehat{u}_{i,k}^{2}dS  \nonumber\\
&\hspace{3cm}
-\frac{\overline{\epsilon}_{k}^{2}}{2}\int_{B_{\delta_{k}}(z_{1,k})}\nabla V(x)\cdot(x-z_{1,k})\widehat{u}_{i,k}^{2}dx
-\overline{\epsilon}_{k}^{2}\int_{B_{\delta_{k}}(z_{1,k})}V(x)\widehat{u}_{i,k}^{2}dx \nonumber\\
&=\frac{\mu_{i,k}\overline{\epsilon}_{k}^{2}}{2}\int_{\partial B_{\delta_{k}}(z_{1,k})}[(x-z_{1,k})\cdot\nu] \widehat{u}_{i,k}^{2}dS
-\mu_{i,k}\overline{\epsilon}_{k}^{2}\int_{B_{\delta_{k}}(z_{1,k})}\widehat{u}_{i,k}^{2}dx \nonumber\\
&\hspace{3cm}  +\frac{a}{4a^\ast}\int_{\partial B_{\delta_{k}}(z_{1,k})}[(x-z_{1,k})\cdot\nu] \widehat{u}_{i,k}^{4}dS
-\frac{a}{2a^\ast}\int_{B_{\delta_{k}}(z_{1,k})}\widehat{u}_{i,k}^{4}dx \nonumber\\
&=\frac{\mu_{i,k}\overline{\epsilon}_{k}^{2}}{2}\int_{\partial B_{\delta_{k}}(z_{1,k})}[(x-z_{1,k})\cdot\nu] \widehat{u}_{i,k}^{2}dS
-\mu_{i,k}\overline{\epsilon}_{k}^{2}\int_{\mathbb R^2}\widehat{u}_{i,k}^{2}dx
+\mu_{i,k}\overline{\epsilon}_{k}^{2}\int_{\mathbb R^2\setminus B_{\delta_{k}}(z_{1,k})}\widehat{u}_{i,k}^{2}dx \nonumber\\
&\hspace{3cm}  +\frac{a}{4a^\ast}\int_{\partial B_{\delta_{k}}(z_{1,k})}[(x-z_{1,k})\cdot\nu] \widehat{u}_{i,k}^{4}dS
-\frac{a}{2a^\ast}\int_{\mathbb R^2}\widehat{u}_{i,k}^{4}dx
+\frac{a}{2a^\ast}\int_{\mathbb R^2\setminus B_{\delta_{k}}(z_{1,k})}\widehat{u}_{i,k}^{4}dx \nonumber\\
&=\frac{\mu_{i,k}\overline{\epsilon}_{k}^{2}}{2}\int_{\partial B_{\delta_{k}}(z_{1,k})}[(x-z_{1,k})\cdot\nu] \widehat{u}_{i,k}^{2}dS
-2a^\ast\overline{\epsilon}_{k}^{4}e_{a}(b_k)
-\frac{b_k}{2a^\ast}\left(\int_{\mathbb R^{2}}|\nabla \widehat{u}_{i,k}|^{2}dx\right)^{2} \nonumber\\
&\quad+\mu_{i,k}\overline{\epsilon}_{k}^{2}\int_{\mathbb R^2\backslash B_{\delta_{k}}(z_{1,k})}\widehat{u}_{i,k}^{2}dx
+\frac{a}{4a^\ast}\int_{\partial B_{\delta_{k}}(z_{1,k})}[(x-z_{1,k})\cdot\nu] \widehat{u}_{i,k}^{4}dS
+\frac{a}{2a^\ast}\int_{\mathbb R^2\backslash B_{\delta_{k}}(z_{1,k})}\widehat{u}_{i,k}^{4}dx.
\end{align}
To simplify the calculations, we denote by $K_1, \ K_2, \ K_3$ and $K_4$ as follows:
\begin{equation*}
\aligned
K_{1}&=\frac{\left(\overline{\epsilon}_{k}^{2}+\frac{b_{k}}{a^{\ast}}\int_{\mathbb R^{2}}|\nabla \widehat{u}_{1,k}|^{2}dx\right)T_{1}-\left(\overline{\epsilon}_{k}^{2}+\frac{b_{k}}{a^{\ast}}\int_{\mathbb R^{2}}|\nabla \widehat{u}_{2,k}|^{2}dx\right)T_{2}}{||\widehat{u}_{1,k}-\widehat{u}_{2,k}||_{L^{\infty}}}\\
&\quad-\frac{\overline{\epsilon}_{k}^{2}}{2}\int_{\partial B_{\delta_{k}}(z_{1,k})}V(x)[(x-z_{1,k})\cdot\nu)](\widehat{u}_{1,k}+\widehat{u}_{2,k}) \widehat{\eta}_{k}dS\\
&\hspace{1cm}+\frac{\mu_{1,k}\overline{\epsilon}_{k}^{2}}{2}\int_{\partial B_{\delta_{k}}(z_{1,k})}[(x-z_{1,k})\cdot\nu](\widehat{u}_{1,k}+\widehat{u}_{2,k}) \widehat{\eta}_{k}dS\\
&\hspace{1.5cm}+\frac{\overline{\epsilon}_{k}^{2}(\mu_{1,k}-\mu_{2,k})}{2||\widehat{u}_{1,k}-\widehat{u}_{2,k}||_{L^{\infty}}}\int_{\partial B_{\delta_{k}}(z_{1,k})}[(x-z_{1,k})\cdot\nu] \widehat{u}_{2,k}^{2}dS\\
&\hspace{2cm}+\frac{a}{4a^\ast}\int_{\partial B_{\delta_{k}}(z_{1,k})}[(x-z_{1,k})\cdot\nu](\widehat{u}_{1,k}^{2}+\widehat{u}_{2,k}^{2})(\widehat{u}_{1,k}+\widehat{u}_{2,k})\widehat{\eta}_{k}dS\\
&\hspace{2.5cm}+\mu_{1,k}\overline{\epsilon}_{k}^{2}\int_{\mathbb R^2\backslash B_{\delta_{k}}(z_{1,k})}(\widehat{u}_{1,k}+\widehat{u}_{2,k}) \widehat{\eta}_{k}dx
+\frac{\overline{\epsilon}_{k}^{2}(\mu_{1,k}-\mu_{2,k})}{||\widehat{u}_{1,k}-\widehat{u}_{2,k}||_{L^{\infty}}}\int_{\mathbb R^2\backslash B_{\delta_{k}}(z_{1,k})}\widehat{u}_{2,k}^{2}dx\\
&\hspace{3.5cm}+\frac{a}{2a^\ast}\int_{\mathbb R^2\backslash B_{\delta_{k}}(z_{1,k})}(\widehat{u}_{1,k}^{2}+\widehat{u}_{2,k}^{2})(\widehat{u}_{1,k}+\widehat{u}_{2,k})\widehat{\eta}_{k}dx,
\endaligned
\end{equation*}
\begin{equation*}
 \hspace{-6cm}K_{2}=\frac{\overline{\epsilon}_{k}^{2}}{2}\int_{B_{\delta_{k}}(z_{1,k})}[\nabla V(x)\cdot x](\widehat{u}_{1,k}+\widehat{u}_{2,k}) \widehat{\eta}_{k}dx,
\end{equation*}
\begin{equation*} \hspace{-5.5cm}
K_{3}=-\frac{\overline{\epsilon}_{k}^{2}}{2}\int_{B_{\delta_{k}}(z_{1,k})}[\nabla V(x)\cdot z_{1,k}](\widehat{u}_{1,k}+\widehat{u}_{2,k}) \widehat{\eta}_{k}dx,
\end{equation*}
\begin{equation*} \hspace{-7cm}
K_{4}=\overline{\epsilon}_{k}^{2}\int_{B_{\delta_{k}}(z_{1,k})}V(x)(\widehat{u}_{1,k}+\widehat{u}_{2,k}) \widehat{\eta}_{k}dx.
\end{equation*}

Then, it follows from \eqref{eq4.44}  and \eqref{eq4.9} that
\begin{equation}\label{eq4.45}
\aligned
K_{1}+K_{2}+K_{3}+K_{4}&=\frac{b_k}{2a^\ast}\int_{\mathbb R^{2}}(|\nabla \widehat{u}_{1,k}|^{2}+|\nabla \widehat{u}_{2,k}|^{2})dx
\int_{\mathbb R^{2}}\nabla(\widehat{u}_{1,k}+\widehat{u}_{2,k})\cdot\nabla \widehat{\eta}_{k}dx\\
&=\frac{b_k}{2a^\ast}\int_{\mathbb R^{2}}(|\nabla \overline{u}_{1,k}|^{2}+|\nabla \overline{u}_{2,k}|^{2})dx
\int_{\mathbb R^{2}}\nabla(\overline{u}_{1,k}+\overline{u}_{2,k})\cdot\nabla \overline{\eta}_{k}dx\\
&=2(1+o(1))b_{k}\int_{\mathbb R^{2}}\nabla Q\cdot\nabla \overline{\eta}_{0}dx,
\endaligned
\end{equation}

Using \eqref{eq4.7}--\eqref{eq4.8}, \eqref{eq4.18} and \eqref{eq4.34*},  similar procedure to that of \eqref{eq4.35}--\eqref{eq4.36}, we can deduce that
\begin{equation}\label{eq4.46}
K_{1}=o( e^{-\frac{C\delta_{k}}{\overline{\epsilon}_{k}}})\ \mbox{as}\ k\rightarrow\infty.
\end{equation}
Moreover, \eqref{eq4.37} indicates that
\begin{equation*}
o( e^{-\frac{C\delta_{k}}{\overline{\epsilon}_{k}}})=\frac{\overline{\epsilon}_{k}^{2}}{2}\int_{B_{\delta_{k}}(z_{1,k})}\frac{\partial V(x)}{\partial x_{j}}(\widehat{u}_{1,k}+\widehat{u}_{2,k})\widehat{\eta}_{k}dx,\ j=1,2,
\end{equation*}
and hence
\begin{equation}\label{eq4.47}
K_{3}=o( e^{-\frac{C\delta_{k}}{\overline{\epsilon}_{k}}})\ \mbox{as}\ k\rightarrow\infty.
\end{equation}
Since $Z_0$ defined by \eqref{eq1.18} has only one point $x_{1}$ by the assumptions of Theorem \ref{th4}, and $\nabla V_{1}(x)\cdot x=pV_{1}(x)$. Then, it follows from \eqref{eq1.23} and \eqref{eq4.37} that
\begin{equation}\label{eq4.48}
\aligned
K_{2}&=\frac{\overline{\epsilon}_{k}^{2}}{2}\int_{B_{\delta_{k}}(z_{1,k})}[\nabla V(x-x_{1}+x_{1})\cdot (x-x_{1})](\widehat{u}_{1,k}+\widehat{u}_{2,k}) \widehat{\eta}_{k}dx\\
&\quad+\frac{\overline{\epsilon}_{k}^{2}}{2}\int_{B_{\delta_{k}}(z_{1,k})}[\nabla V(x)\cdot x_{1}](\widehat{u}_{1,k}+\widehat{u}_{2,k}) \widehat{\eta}_{k}dx\\
&=\frac{\overline{\epsilon}_{k}^{2}}{2}\int_{B_{\delta_{k}}(z_{1,k})}\Big\{[\nabla V_{1}(x-x_{1})+W(x-x_{1})]\cdot (x-x_{1})\Big\}(\widehat{u}_{1,k}+\widehat{u}_{2,k}) \widehat{\eta}_{k}dx+o( e^{-\frac{C\delta_{k}}{\overline{\epsilon}_{k}}})\\
&=\frac{p(1+o(1))}{2}\overline{\epsilon}_{k}^{p+4}\int_{B_{\frac{\delta_{k}}{\overline{\epsilon}_{k}}}(0)}
V_{1}\left(x+\frac{z_{1,k}-x_{1}}{\overline{\epsilon}_{k}}\right)(\overline{u}_{1,k}+\overline{u}_{2,k}) \overline{\eta}_{k}dx+o( e^{-\frac{C\delta_{k}}{\overline{\epsilon}_{k}}})\\
&=p(1+o(1))\overline{\epsilon}_{k}^{p+4}\int_{\mathbb R^2}
V_{1}(x+y_{0})Q \overline{\eta}_{0}dx,
\endaligned
\end{equation}
and
\begin{equation}\label{eq4.49}
\aligned
K_{4}&=\overline{\epsilon}_{k}^{4}\int_{B_{\frac{\delta_{k}}{\overline{\epsilon}_{k}}}(0)}
\frac{V(\overline{\epsilon}_{k}x+z_{1,k}-x_{1}+x_{1})}{V_{1}(\overline{\epsilon}_{k}x+z_{1,k}-x_{1})}V_{1}(\overline{\epsilon}_{k}x+z_{1,k}-x_{1})
(\overline{u}_{1,k}+\overline{u}_{2,k}) \overline{\eta}_{k}dx\\
&=\overline{\epsilon}_{k}^{p+4}\int_{B_{\frac{\delta_{k}}{\overline{\epsilon}_{k}}}(0)}
\frac{V(\overline{\epsilon}_{k}x+z_{1,k}-x_{1}+x_{1})}{V_{1}(\overline{\epsilon}_{k}x+z_{1,k}-x_{1})}
V_{1}\left(x+\frac{z_{1,k}-x_{1}}{\overline{\epsilon}_{k}}\right)
(\overline{u}_{1,k}+\overline{u}_{2,k}) \overline{\eta}_{k}dx\\
&=2(1+o(1))\overline{\epsilon}_{k}^{p+4}\int_{\mathbb R^2}V_{1}(x+y_{0})Q \overline{\eta}_{0}dx,
\endaligned
\end{equation}
where $W(x)=(W_{1}(x),W_{2}(x))$.

Therefore, $d_{0}=0$ in \eqref{eq4.15} and $h_{0}=\overline{h}_{0}=0$ in \eqref{eq4.15*} can be obtained by considering the cases of $a=a^{\ast}$ and $a>a^{\ast}$, respectively.

If $a=a^\ast$, it follow from \eqref{eq1.14}, \eqref{eq1.21} and \eqref{eq4.10} that
\begin{align}\label{eq4.50}
& \int_{\mathbb R^2}V_{1}(x+y_{0})Q \overline{\eta}_{0}dx
=\int_{\mathbb R^2}V_{1}(x+y_{0})Q\left[d_{0}(Q+x\cdot\nabla Q)+\sum_{i=1}^{2}d_{i}\frac{\partial Q}{\partial x_{i}}\right]dx \nonumber\\
&=d_{0}\int_{\mathbb R^2}V_{1}(x+y_{0})Q(Q+x\cdot\nabla Q)dx+\sum_{i=1}^{2}\frac{d_{i}}{2}\int_{\mathbb R^2}V_{1}(x+y_{0})\frac{\partial Q^{2}}{\partial x_{i}}dx \nonumber\\
&=d_{0}\int_{\mathbb R^2}V_{1}(x+y_{0})Q^{2}dx+\frac{d_{0}}{2}\int_{\mathbb R^2}V_{1}(x+y_{0})(x\cdot\nabla Q^{2})dx -\sum_{i=1}^{2}\frac{d_{i}}{2}\int_{\mathbb R^2}\frac{\partial V_{1}(x+y_{0})}{\partial x_{i}}Q^{2}dx \nonumber\\
&=d_{0}\int_{\mathbb R^2}V_{1}(x+y_{0})Q^{2}dx-\frac{d_{0}}{2}\int_{\mathbb R^2}Q^{2}[2V_{1}(x+y_{0})+x\cdot\nabla V_{1}(x+y_{0})]dx \nonumber\\
&=-\frac{pd_{0}}{2}\int_{\mathbb R^2}V_{1}(x+y_{0})Q^{2}dx+\frac{d_{0}}{2}\int_{\mathbb R^2}Q^{2}[y_{0}\cdot\nabla V_{1}(x+y_{0})]dx =-\frac{pd_{0}\lambda_{0}}{2},
\end{align}
where $\lambda_{0}$ is given by \eqref{eq1.17}. Moreover, \eqref{eq4.10} implies that
\begin{equation}\label{eq4.51*}
\aligned
&\int_{\mathbb R^{2}}\nabla Q\cdot\nabla \overline{\eta}_{0}dx
=\int_{\mathbb R^{2}}\nabla Q\cdot\nabla \left[d_{0}(Q+x\cdot\nabla Q)+\sum_{i=1}^{2}d_{i}\frac{\partial Q}{\partial x_{i}}\right]dx\\
&=d_{0}\int_{\mathbb R^{2}}|\nabla Q|^{2}dx+d_{0}\int_{\mathbb R^{2}}\nabla Q\cdot\nabla(x\cdot\nabla Q)dx
+\sum_{i=1}^{2}d_{i}\int_{\mathbb R^{2}}\nabla Q\cdot\nabla\frac{\partial Q}{\partial x_{i}}dx\\
&=a^\ast d_{0}.
\endaligned
\end{equation}
Therefore, for the case of $a=a^\ast$, we see from \eqref{eq1.20}, \eqref{eq4.45}--\eqref{eq4.51*} that
\begin{equation}\label{eq4.52}
\aligned
o( e^{-\frac{C\delta_{k}}{\overline{\epsilon}_{k}}})&=\frac{b_k}{2a^\ast}\int_{\mathbb R^{2}}(|\nabla \widehat{u}_{1,k}|^{2}+|\nabla \widehat{u}_{2,k}|^{2})dx
\int_{\mathbb R^{2}}\nabla(\widehat{u}_{1,k}+\widehat{u}_{2,k})\cdot\nabla \widehat{\eta}_{k}dx-K_{2}-K_{4}\\
&=pd_{0}\lambda_{0}(1+o(1))\overline{\epsilon}_{k}^{p+4}
+\frac{p^{2}d_{0}\lambda_{0}}{2}(1+o(1))\overline{\epsilon}_{k}^{p+4}
+pd_{0}\lambda_{0}(1+o(1))\overline{\epsilon}_{k}^{p+4}\\
&=\frac{pd_{0}\lambda_{0}(p+4)}{2}(1+o(1))\overline{\epsilon}_{k}^{p+4}.
\endaligned
\end{equation}

On the other hand, if $a>a^\ast$, it follows from \eqref{eq1.14}, \eqref{eq1.21} and \eqref{eq4.10*} that
\begin{align}\label{eq4.50*}
&\int_{\mathbb R^2}V_{1}(x+y_{0})Q \overline{\eta}_{0}dx
=\int_{\mathbb R^2}V_{1}(x+y_{0})Q\left(h_{0}Q+\overline{h}_{0}x\cdot\nabla Q+\sum_{i=1}^{2}h_{i}\frac{\partial Q}{\partial x_{i}}\right)dx \nonumber \\
& =h_{0}\int_{\mathbb R^2}V_{1}(x+y_{0})Q^{2}dx+\frac{\overline{h}_{0}}{2}\int_{\mathbb R^2}V_{1}(x+y_{0})(x\cdot\nabla Q^{2})dx 
-\sum_{i=1}^{2}\frac{h_{i}}{2}\int_{\mathbb R^2}\frac{\partial V_{1}(x+y_{0})}{\partial x_{i}}Q^{2}dx \nonumber \\
&=h_{0}\int_{\mathbb R^2}V_{1}(x+y_{0})Q^{2}dx-\frac{\overline{h}_{0}}{2}\int_{\mathbb R^2}Q^{2}[2V_{1}(x+y_{0})+x\cdot\nabla V_{1}(x+y_{0})]dx \nonumber \\
&=(h_{0}-\overline{h}_{0})\int_{\mathbb R^2}V_{1}(x+y_{0})Q^{2}dx-\frac{\overline{h}_{0}}{2}\int_{\mathbb R^2}Q^2[(x+y_{0})\cdot\nabla V_{1}(x+y_{0})]dx \nonumber \\
&\hspace{9cm}
+\frac{\overline{h}_{0}}{2}\int_{\mathbb R^2}Q^{2}[y_{0}\cdot\nabla V_{1}(x+y_{0})]dx \nonumber \\
&=\frac{2h_{0}-2\overline{h}_{0}-p\overline{h}_{0}}{2}\int_{\mathbb R^2}V_{1}(x+y_{0})Q^{2}dx
=\frac{(2h_{0}-2\overline{h}_{0}-p\overline{h}_{0})\lambda_{0}}{2}.
\end{align}
Moreover, \eqref{eq4.10*} implies that
\begin{align}\label{eq4.51**}
&\int_{\mathbb R^{2}}\nabla Q\cdot\nabla \overline{\eta}_{0}dx
=\int_{\mathbb R^{2}}\nabla Q\cdot\nabla \left[h_{0}Q+\overline{h}_{0}x\cdot\nabla Q+\sum_{i=1}^{2}h_{i}\frac{\partial Q}{\partial x_{i}}\right]dx \nonumber\\
&=h_{0}\int_{\mathbb R^{2}}|\nabla Q|^{2}dx+\overline{h}_{0}\int_{\mathbb R^{2}}\nabla Q\cdot\nabla(x\cdot\nabla Q)dx
+\sum_{i=1}^{2}h_{i}\int_{\mathbb R^{2}}\nabla Q\cdot\nabla\frac{\partial Q}{\partial x_{i}}dx
=a^\ast h_{0}.
\end{align}
Hence, for the case of $a>a^\ast$, we see from \eqref{eq1.20}, \eqref{eq4.45}--\eqref{eq4.49} and  \eqref{eq4.50*}--\eqref{eq4.51**} that
\begin{equation}\label{eq4.53}
\aligned
o( e^{-\frac{C\delta_{k}}{\overline{\epsilon}_{k}}})&=\frac{b_k}{2a^\ast}\int_{\mathbb R^{2}}(|\nabla \widehat{u}_{1,k}|^{2}+|\nabla \widehat{u}_{2,k}|^{2})dx
\int_{\mathbb R^{2}}\nabla(\widehat{u}_{1,k}+\widehat{u}_{2,k})\cdot\nabla \widehat{\eta}_{k}dx-K_{2}-K_{4}\\
&=2h_{0}(a-a^\ast)(1+o(1))\overline{\epsilon}_{k}^{2}
-\frac{(2h_{0}-2\overline{h}_{0}-p\overline{h}_{0})p\lambda_{0}}{2}(1+o(1))\overline{\epsilon}_{k}^{p+4}\\
&\quad\quad-(2h_{0}-2\overline{h}_{0}-p\overline{h}_{0})\lambda_{0}(1+o(1))\overline{\epsilon}_{k}^{p+4}\\
&=2h_{0}(a-a^\ast)(1+o(1))\overline{\epsilon}_{k}^{2}.
\endaligned
\end{equation}

So, $d_{0}=h_{0}=0$ follows from \eqref{eq4.52} and \eqref{eq4.53}.

Moreover, for the case of $a>a^\ast$, it follows from \eqref{eq4.14*}, \eqref{eq4.51**} and $h_{0}=0$ that $\overline{\eta}_{0}$ satisfies \eqref{eq4.14}. Hence, $\overline{\eta}_{0}$ must be of the form of \eqref{eq4.10} and $h_{0}=\overline{h}_{0}=0$.
Particularly, using \eqref{eq4.15} and \eqref{eq4.15*}, we know that
\begin{equation*}
\sum_{i=1}^{2}d_{i}\int_{\mathbb R^{2}}\frac{\partial^{2} V_{1}(x+y_{0})}{\partial x_{i}\partial x_{j}}Q^{2}dx=0,
\end{equation*}
and
\begin{equation*}
\sum_{i=1}^{2}h_{i}\int_{\mathbb R^{2}}\frac{\partial^{2} V_{1}(x+y_{0})}{\partial x_{i}\partial x_{j}}Q^{2}dx=0.
\end{equation*}
This indicate that $d_{1}=d_{2}=0$ and $h_{1}=h_{2}=0$ due to the non-degeneracy assumption \eqref{eq1.21}, hence $\overline{\eta}_{0}\equiv0$ for $a\geq a^\ast$.

On the other hand, if $a\geq a^\ast$, we claim that $\overline{\eta}_{0}\equiv0$ can not occur. Indeed, let $y_{k}\in\mathbb R^2$ be the maximum point of $\overline{\eta}_{k}$, where $\overline{\eta}_{k}(y_{k})=||\overline{\eta}_{k}||_{L^{\infty}}=1.$ Applying the maximum principle to \eqref{eq4.5}, we see that $|y_{k}|\leq C$ for all $k$ due to the exponential decay as \eqref{eq4.7}. Therefore, \eqref{eq4.9} implies that $\overline{\eta}_{0}\not\equiv0$ on $\mathbb R^2$. So, our assumption that $u_{1,k}\not\equiv u_{2,k}$ is false, and we complete the proof of Theorem \ref{th4}.
\end{proof}

\end{document}